%% file: IS_SR_Levy_processes.tex
\documentclass[12pt]{article}
\usepackage[latin1]{inputenc}
\usepackage[english]{babel}
\usepackage{amsmath,amssymb,amsthm}
\usepackage[dvips]{graphicx}
\usepackage{graphicx}
\usepackage{fullpage}
\usepackage{float}
\usepackage{varioref}
\usepackage[active]{srcltx}
\usepackage[square,sort,comma,numbers]{natbib}
\usepackage{mathrsfs}
\usepackage{authblk}
\usepackage{xcolor}
\usepackage{epstopdf}
\usepackage{color}
\usepackage{caption}
\usepackage{multirow}
\usepackage{enumerate}
\usepackage{booktabs}
\usepackage{stmaryrd}
\usepackage{url}
\usepackage{graphicx}
\usepackage{epstopdf}
\usepackage{verbatim}
\usepackage{appendix}
\usepackage{algorithm}
\usepackage{algorithmic}

\definecolor{refkey}{gray}{.75}
\definecolor{labelkey}{gray}{.75}
\usepackage[dvips]{graphicx}
\def\RR{\mathbb   R}
\def\EE{\mathbb   E}
\def\NN{\mathbb   N}
\def\PP{\mathbb   P}

\def\1{{\bf 1}}

\newcommand{\V}{\mathrm{Var}}
\newcommand{\veps}{\upsilon_\varepsilon}

\newcommand{\Tr}{\mathrm{Tr}}

\newtheorem{theorem}{Theorem}[section]
\newtheorem{proposition}{Proposition}[section]
\newtheorem{lemma}{Lemma}[section]
\newtheorem{corollary}{Corollary}[section]
\newtheorem{definition}{Definition}[section]
\newtheorem{remark}{Remark}
\newcommand{\argmin}{\mathop{\mathrm{arg\,min}}}
\newcommand\independent{\protect\mathpalette{\protect\independenT}{\perp}}

\def\independenT#1#2{\mathrel{\rlap{$#1#2$}\mkern2mu{#1#2}}}

\begin{document}
\title{\textsc{Importance Sampling and Statistical Romberg Method for L\'evy Processes}}
\author[1]{\textsc{Mohamed Ben Alaya}\thanks{Supported by Laboratory of Excellence MME-DII http://labex-mme-dii.u-cergy.fr/}}
\author[1]{\textsc{Kaouther Hajji}}
\author[1]{\textsc{Ahmed Kebaier}\thanks{This research benefited from the support of the chair "Risques Financiers", Fondation du Risque.}\thanks{Supported by Laboratory of Excellence MME-DII http://labex-mme-dii.u-cergy.fr/}}
\affil[1]{{\normalsize{ Universit\'{e} Paris 13,
Sorbonne Paris Cité, LAGA, CNRS (UMR 7539)}}
\normalsize{mba@math.univ-paris13.fr\hspace{0.2 cm}
hajji@math.univ-paris13.fr\hspace{0.2cm}kebaier@math.univ-paris13.fr}}
\maketitle

\begin{abstract}
An important family of stochastic processes arising in many areas of applied probability is the class of Lévy processes. Generally, such processes are not simulatable especially for those with infinite activity.
In practice, it is common to approximate them by truncating the jumps at some cut-off size $\varepsilon$ 
($\varepsilon\searrow 0$). This procedure leads us to consider a simulatable compound Poisson process. 
This paper first introduces, for this setting, the statistical Romberg method to improve the complexity of the classical Monte Carlo one. Roughly speaking, we use many
sample paths with a coarse cut-off $\varepsilon^{\beta},$ $\beta\in(0,1)$, and few additional sample paths
with a fine cut-off $\varepsilon$. Central limit theorems of Lindeberg-Feller type  for both
Monte Carlo and statistical Romberg method for the inferred errors depending on the parameter $\varepsilon$
are proved. This leads to an accurate description of the optimal choice of parameters with explicit limit variances.  Afterwards, the authors propose a stochastic approximation method of finding the optimal measure change by Esscher transform for Lévy processes with Monte Carlo and statistical Romberg importance sampling variance reduction. Furthermore, we develop new adaptive Monte Carlo and  statistical Romberg algorithms and prove the associated central limit theorems. Finally, numerical simulations are  processed to illustrate the efficiency of the adaptive statistical Romberg method that reduces at the same time the variance and the computational effort 
associated to the effective computation of option prices when the
underlying asset process follows an exponential pure jump  CGMY model.


\end{abstract}

\section{Introduction}
Lévy processes arise in many areas of applied probability and specially in mathematical finance, where they  become very fashionable since they can describe the observed reality of financial markets in a more accurate way than models based on Brownian motion (see e.g. Cont and Tankov \cite{ContTankov} and 
Shoutens \cite{Schoutens}). In particular in the pricing of financial securities we are interested in the computation of the real quantity  $\EE F(L_T)$, $T>0$, where $(L_t)_{0\leq t \leq T }$ is a  $\RR^d$-valued  pure jump L\'evy process, $d\geq 1$ and $F:\RR^d\mapsto \RR$ is a given function. In the literature, the computation of this quantity involves three types of methods: Fourier transform methods, numerical methods for partial integral differential equations and Monte Carlo methods. It is well known that the two first methods can not cope with high dimensional problems.
This gives a competitive edge for Monte Carlo methods in this setting. Therefore, the focus of this work is
to study improved Monte Carlo methods using the statistical Romberg  algorithm and the importance sampling
 technique. The statistical Romberg method is known for reducing the time complexity and the  importance sampling technique is aimed at reducing the variance.

The Monte Carlo method consists of two steps. In the first step, we approximate the Lévy process $(L_t)_{0\leq t\leq T}$ by a simulatable Lévy process  $(L^{\varepsilon}_t)_{0\leq t\leq T}$ with $\varepsilon>0$. If $\nu$ denotes the Lévy measure of the Lévy process under consideration, then it is common to take $(L^{\varepsilon}_t)_{0\leq t\leq T}$ with Lévy measure $\nu_{|\{|x|\geq \varepsilon\}} $ and $\varepsilon \searrow 0$. This approximation 
is nothing but a compound Poisson process. 
In the second step, we approximate $\displaystyle\mathbb E\,F\left( L^{\varepsilon}_T\right)$ by
$
\frac{1}{N}\sum_{i=1}^{N}F( L^{\varepsilon}_{T,i}),
$
where $(L^{\varepsilon}_{T,i})_{1\leq i\leq N}$ is a sample of $N$ independent copies of $L^{\varepsilon}_T$.
Therefore, this Monte Carlo method (MC) is affected respectively 
 by an approximation error and a  statistical one
$$
\mathcal E_1(\varepsilon):=\mathbb E\left( F(L^{\varepsilon}_T)-F(L_T)\right)
\mbox{ and  } \mathcal E_2(N):=\frac{1}{N}\sum_{i=1}^{N}F( L^{\varepsilon}_{T,i})-\mathbb EF(L^{\varepsilon}_T).
$$
On one hand, for a Lipschitz function $F$ we have $\mathcal E_1(\varepsilon)=O(\sigma(\varepsilon)), $ where
$\sigma^2(\varepsilon)=\EE|L_1-L^{\varepsilon}_1|^2$ (see relation \eqref{weak_error_upperbound} for more details). On the other hand, the statistical error is controlled by the central limit theorem  
with order $1/\sqrt{N}$. Hence, optimizing the choice  of the sample size $N$ in the Monte Carlo method leads to  $N=O(\sigma^{-2}(\varepsilon))$.  Moreover, if we choose $N=\sigma^{-2}(\varepsilon)$ we prove a central 
limit theorem of Lindeberg-Feller type (see Theorem \ref{clt:MC}).
Therefore, if we denote by  $\mathcal K(\varepsilon)$ the cost of a single 
simulation of $L^{\varepsilon}_T$, then the total time complexity necessary to achieve the  precision $\sigma(\varepsilon)$ is given by  $C_{MC}=O(\mathcal K(\varepsilon)\sigma^{-2}(\varepsilon))$ (see subsection \ref{complexity}). 

In order to improve the performance of this method we use the idea of the statistical Romberg method introduced by Kebaier \cite{Kebaier} in the setting of Euler Monte Carlo methods for stochastic differential equations driven by a standard Brownian Motion which is also related to the well known Romberg's method introduced by Talay and Tubaro in \cite{TalayTubaro}.
Inspired by this technique, we introduce a novel method for the computation of our initial target. 
The main idea of this new method is to consider two cut-off sizes  $\varepsilon$ and $\varepsilon^{\beta}$, $\beta\in(0,1)$
and then approximate $\EE F(L_T)$ by 
\begin{equation*}
 \frac{1}{N_1}\sum_{i=1}^{N_1}
F(\hat L^{\varepsilon^{\beta}}_{T,i}) + \frac{1}{N_2}
\sum_{i=1}^{N_2}F( L^{\varepsilon}_{T,i})-F( L^{\varepsilon^{\beta}}_{T,i}).
\end{equation*}
The samples $(L^{\varepsilon}_{T,i})_{1\leq i\leq N_2}$ 
and $(L^{\varepsilon^{\beta}}_{T,i})_{1\leq i\leq N_2}$ have to be independent of  
 $ (\hat L^{\varepsilon}_{T,i})_{1\leq i\leq N_1}$. Moreover,  for  $1\leq i\leq N_2$, the process
$(L^{\varepsilon}_{t,i})_{0\leq t\leq T}$ is nothing else the sum of $(L^{\varepsilon^{\beta}}_{t,i})_{0\leq t\leq T}$ and an independent Lévy process $(L^{\varepsilon,\varepsilon^{\beta}}_{t,i})_{0\leq t\leq T}$   with Lévy measure $\nu_{|\{\varepsilon \leq |x|\leq\varepsilon^{\beta}\}}$ which is also simulatable as a compound Poisson process. This new method  will be referred as the statistical Romberg method (SR). 
Additionally, like for the MC method, we prove a central limit theorem of Lindeberg-Feller type for the SR algorithm with  $N_1=\sigma^{-2}(\varepsilon)$ and $N_2=\sigma^{-2}(\varepsilon)\sigma^{2}(\varepsilon^{\beta})$ (see Theorem \ref{CLT Statistical Romberg}). Then, according to subsection \ref{complexity}, the total time complexity necessary to achieve the  precision $\sigma(\varepsilon)$ is given by $C_{SR}=\left(\mathcal K(\varepsilon^{\beta}) +\mathcal K(\varepsilon)\sigma^{2}(\varepsilon^{\beta})\right)\sigma^{-2}(\varepsilon)$. It turns out that the complexity ratio  $C_{SR}/C_{MC}$ vanishes as $\varepsilon$ goes to zero.

Since the efficiency of the Monte Carlo simulation considerably depends on the smallness of the variance in the estimation, many variance reduction techniques were developed in the recent years. Among these methods appears the technique of importance sampling very popular for its efficiency. For the Gaussian setting, the importance sampling technique was studied by Arouna  \cite{Arouna}, Galasserman, Heidelberger and Shahabuddin \cite{GlassermanheidelbergerShahabuddin} for MC method and by Ben Alaya, Hajji and Kebaier \cite{BenalayaHajjiKebaier} for SR method. Concerning Lévy process without a Brownian component, Kawai \cite{Kawai} has already applied this technique for MC algorithm using the Esscher transform which is nothing but the well known exponential tilting of laws. From a practical point of view, his approach is exploitable only when the Lévy process $(L_t)_{0\leq t\leq T}$ is simulatable without any approximation. Note also that in his study there is no results on the rate of convergence of the obtained algorithm.  

The main aim of the present work is to apply the idea of \cite{Kawai} to the approximation Lévy process
$(L_t^{\varepsilon})_{0\leq t\leq T}$ for both MC and SR algorithms and to study the inferred error in terms of the cut-off $\varepsilon$; a question which has not been addressed in previous research.  
Roughly speaking, thanks to the  Esscher transform we produce a parametric transformation such that for all $\theta\in K$  we have $\EE F(L^{\varepsilon}_T)=\EE G(\theta,L^{\varepsilon}_T),$  
where $K$ is a suitable subset of $\mathbb R^d$ and $(\theta,x)\mapsto G(\theta,x)$ is a real function taking values in $\RR^d\times\RR^d$.
Concerning the MC method it looks natural to implement the method with 
$\theta^*_{1,\varepsilon}=\argmin_{\theta\in K}\EE G^2(\theta,L^{\varepsilon}_T).$
However, for the SR method the inferred error is controlled by 
${ \rm Var}(G(\theta,L_T^{\varepsilon}))+T\EE(|\nabla_x G(\theta,L_T^{\varepsilon})|^2)$. 
Then, in this case, it is natural to implement the first (resp. the second) empirical mean appearing in the SR estimator with $ \theta^*_{1,\varepsilon}$ (resp. 
$\theta^*_{2,\varepsilon}=\argmin_{\theta\in K}\EE(|\nabla_x G(\theta,L_T^{\varepsilon})|^2)$. 
But what about the effective computation of $(\theta^*_{i,\varepsilon})_{i\in\{1,2\}}$ ? 
To answer this question, we use a constrained version of the well-known stochastic approximation Robbins-Monro. All these ideas led us to introduce two new  methods based on  adaptive approximations. 
The first method concerns a combination of an adaptive  importance sampling technique and the 
MC method that will be called Importance Sampling Monte Carlo method (ISMC) (see relation \eqref{eq:ISMC}). 
The second one concerns an original combination of an adaptive  importance sampling technique with the SR algorithm that will be referred as Importance Sampling Statistical Romberg method (ISSR) (see relation \eqref{eq:ISSR}). The main point in favor of the ISSR method is that it inherits the variance reduction 
from the Importance sampling procedure and the complexity reduction from the SR method. A complexity analysis is also provided.

The rest of the paper is organized as follows. Section \ref{GF} introduces the general framework and recalls some useful results. In section \ref{SRLevy}, the central limit theorems of Lindeberg-Feller type are proved for both MC and SR methods (see Theorems \ref{clt:MC} and \ref{CLT Statistical Romberg}).  Similar results are derived for the setting of an exponential Lévy model (see Corollaries \ref{Cor-CLT-MC} and \ref{Cor-CLT-RS}).  A complexity analysis is included. In section \ref{sec:IS}, we recall the Esscher transform and the principle of importance sampling technique for the SR method. For $i\in\{1,2\}$ and $\varepsilon \searrow 0$, we prove the convergence of the optimal choice $\theta^*_{i,\varepsilon}$  to the optimal choice associated to the limit model (see Theorem \ref {th:convergence}). In section \ref{adaptativeprocedure}, we first study, for $i\in\{1,2\}$, the almost sure convergence of the stochastic recursive constrained Robbins-Monro algorithm given by the double indexed sequence $\theta_{i,\varepsilon,n}$ as $\varepsilon\searrow 0$ and $n\nearrow \infty$  (see Theorems \ref{th:cv:CRM} and  \ref{th:cv:eps} and Corollary \ref{cor:cv:CRM}). The rest of this section is devoted to 
prove the central limit theorems of Lindeberg-Feller type for both adaptive ISMC and ISSR methods
(see Theorems \ref{CLT adaptative Monte Carlo} and \ref{CLT:ISSR}). Section \ref{num:test} illustrates 
the superiority of the ISSR method over all the other ones via numerical examples for both one and two-dimensional Carr, Geman, Madan and Yor (CGMY) process \cite{CGMY}. Finally, the last Section is devoted to discuss some future openings.

\section{General Framework}\label{GF}
We denote by $(\Omega,\mathcal F,\PP)$ our underlying probability space. 
A stochastic process $(L_t)_{t\geq 0}$ on  $(\Omega,\mathcal F,\PP)$ with values in $\RR^{d}$ such that $L_0=0$ is a L\'evy process if it has independent and stationary increments. We endow the probability space 
$(\Omega,\mathcal F,\PP)$ with the canonical filtration $(\mathcal F_t)_{0\leq t\leq T}$ where $\mathcal F_t=\sigma(L_s,s\leq t)$. 
The characteristic function of a L\'evy process $L$ with generating triplet $(\gamma, A, \nu)$ is given by the well known L\'evy Kintchine representation
\begin{equation*}
 \EE e^{iu.L_t}  = \exp \left\{ t \left( i \gamma.u - \frac{1}{2} u.Au + \int_{\RR^{d}} (e^{iu.x} -1-iu.x \mathbf{1}_{|x|\leq 1}) \nu(dx) \right) \right\},  \quad u \in \RR^{d}, 
\end{equation*}
where $\gamma \in \RR^{d}$, $A$ is a symmetric non-negative-definite $d \times d $ matrix and $\nu$ is a L\'evy measure on $\RR^{d}\setminus \{0\}$ verifying $\int_{\RR^{d}\setminus \{0\}} (|x|^{2} \wedge 1) \nu(dx) < \infty$.   
(Given vectors $x$ and $y\in \RR^d$, $x.y$ denotes the inner product of $x$ and $y$ associated to the Euclidean norm $|\cdot|$). In this paper, we are interested in studying pure-jump L\'evy processes, that is, we set $A \equiv 0$ throughout all the paper.Then, $(L_{t})_{t \geq 0}$ is a L\'evy process with generating triplet $(\gamma, 0, \nu)$. 
The simulation of a L\'evy process with infinite L\'evy measure is not straightforward. 
From the L\'evy-It\^o decomposition (see e.g. Theorem 19.2 in Sato \cite{Sato}), we know that $L$ can be represented as a sum of a compound 
Poisson process  and an almost sure limit of compensated compound Poisson process $L_t=\lim_{\varepsilon \rightarrow 0}L_t^{\varepsilon}$ $a.s.$
where  for $0<\varepsilon<1$
\begin{equation}
\label{Levy_app}
 L_{t}^{\varepsilon}=\gamma t+  \sum_{0 < s \leq t}  \Delta L_s \mathbf{1}_{|\Delta L_s| > 1}+
 (\sum_{0 < s \leq t}\Delta L_s \mathbf{1}_{\varepsilon\leq |\Delta L_s|\leq 1 }-t\int_{\varepsilon \leq |x| \leq 1}x\nu(dx)),  \quad t\geq 0.
\end{equation}
Note that without the compensation $t\int_{\varepsilon \leq |x| \leq 1}x\nu(dx)$, the sum of jumps $\sum_{0 < s \leq t}\Delta L_s \mathbf{1}_{\varepsilon\leq |\Delta L_s|\leq 1 }$ may not converge as $\varepsilon$ goes to zero. We denote  the approximation error by
\begin{equation}
 R^{\varepsilon}= L - L^{\varepsilon}.
 \label{error}
\end{equation}
The process $R^{\varepsilon}$ is also a L\'evy process independent of $L^{\varepsilon}$ with characteristic function
\begin{equation*}
 \EE e^{iu.R_{t}^{\varepsilon}}= \exp \left\{ t \int_{|x|\leq \varepsilon} (e^{iu.x} -1 - iu.x) \nu(dx) \right\}. 
\end{equation*}
Consequently, $\EE [R_{t}^{\varepsilon}]=0$ and the variance-covariance matrix $\EE [R_{t}^{\varepsilon}{(R_{t}^\varepsilon})']= t \Sigma_\varepsilon$ where 
$$\Sigma_\varepsilon= \int_{|x|\leq \varepsilon} xx' \nu(dx) .$$
($A'$ denotes the transpose of a matrix $A$).
The asymptotic behavior of the distribution of $R^{\varepsilon}$ is firstly studied by Asmussen and Rosi\'nski \cite{AsmRsi} in the one dimensional case and later extended to
the multidimensional case by Cohen and Rosi\'nski \cite{CohRos}. Throughout this paper  $W=(W_t)_{t\geq 0}$ is a standard Brownian motion in $\RR^d$ independent of $(L_t)_{t\geq 0}$.
\begin{theorem}\label{Cohen_Rosinski}Under the above notation, suppose that $\Sigma_\varepsilon$ is invertible for every $\varepsilon \in (0,1]$. Then as $\varepsilon \rightarrow 0$,
$$
\Sigma_\varepsilon^{-1/2} R^{\varepsilon}{\Rightarrow} W,
$$ 
 if and only if for each $k>0$
\begin{equation}
\label{condition_convergence1}
\lim\limits_{\varepsilon \to 0} \int_{\langle \Sigma_\varepsilon^{-1}x,x\rangle > k}  \langle \Sigma_\varepsilon^{-1}x,x\rangle {\mathbf 1}_{|x|\leq \varepsilon} \nu(dx)= 0.
\end{equation}
Here $``{\Rightarrow}``$ stands for the convergence in distribution.
\end{theorem}
If $\nu$ is given in polar coordinates by $\nu(dr,du)=\mu(dr|u)\lambda(du), \quad r>0, u\in S^{d-1},$
where $\{\mu(\cdot|u): u\in S^{d-1}\}$ is a measurable family of L\'evy measures on $(0,\infty)$ and $\lambda$ is a 
finite measure on the unit sphere $S^{d-1}$, then $$\Sigma_\varepsilon= \int_{S^{d-1}} \int_{0}^{\varepsilon} r^2 uu'\mu(dr|u)\lambda(du).$$
If we define 
$
 \sigma^2(\varepsilon,u):=\int_{0}^{\varepsilon} r^2 \mu(dr|u)
$
and  $\sigma^2(\varepsilon):=\int_{S^{d-1}}\sigma^2(\varepsilon,u)\lambda(du)$, then  
\begin{equation}
\label{eq-rate-Sigmaeps}
 \EE  |L_t - L_t^{\varepsilon}|^2=t\Tr(\Sigma_{\varepsilon})=t\sigma^2(\varepsilon). 
\end{equation}

\begin{remark} In the one dimensional case Assmussen and Rosi\'nski \cite{AsmRsi}  have obtained the convergence of $\sigma^{-1}(\varepsilon)R^{\varepsilon}$ to a standard Brownian motion
if and only if for each $k>0$, $\sigma(k\sigma(\varepsilon) \wedge \varepsilon  )\sim \sigma(\varepsilon)$ which is satisfied as soon as 
$\lim\limits_{\varepsilon \to 0} \frac{\sigma(\varepsilon)}{\varepsilon}=\infty$ (see Theorem 2.1 and Proposition 2.1 in \cite{AsmRsi}). 
An extension to this sufficient condition in the multidimensional case is given by Theorem 2.5 in Cohen and Rosi\'nski \cite{CohRos}. 
Suppose that  the support of the  measure  $\lambda$ is not   contained in any proper linear subspace of $\RR^d$,  they proved that if 
  \begin{equation}
  \lim\limits_{\varepsilon \to 0} \frac{\sigma(\varepsilon,u)}{\varepsilon}=\infty, \lambda-a.e.
\label{condition_convergence2}
  \end{equation} 
  then $\Sigma_{\varepsilon}$ is invertible and condition (\ref{condition_convergence1}) of Theorem \ref{Cohen_Rosinski} holds. 
\end{remark}
On the other hand, according to Proposition 2.1 of Dia \cite{Dia13}, we have a $L^q$-upper bound of the error approximation in the one dimensional case for any real $q >0$. This result on the strong error approximation remains valid for the  multidimensional case. More precisely,  if we consider the $d$-dimensional error L\'evy process  $R^{\varepsilon}$ given by relation (\ref{error}), then we can easily deduce that 
\begin{equation}
\label{bounds_error_dia}
\EE|R_{t}^{\varepsilon}|^{q} \leq K_{q,T} \sigma_{0}(\varepsilon)^{q},\quad \mbox{ where } K_{q,T}>0 \mbox{ and } \sigma_{0}(\varepsilon)=\sigma(\varepsilon) \vee \varepsilon. \tag{SE}
\end{equation}
 Concerning the weak error, if $F$ denotes a real valued Lipschitz continuous function with Lipschitz constant  $C>0$,  then  it is easy to see that
\begin{equation}
\label{weak_error_upperbound}
|\EE F(L_T)- \EE F(L_T^{\varepsilon})| \leq C \sqrt{T}\sigma(\varepsilon)
\end{equation}
Moreover, under some regularity conditions on   function $F$ we can obtain an expansion of the weak error as in Proposition 2.2 and  Remark 2.3  of \cite{Dia13}.  So, it is worth to introduce the following assumption: there exist $C_F\in\RR$ and $\upsilon_\varepsilon \searrow 0$   as $\varepsilon \searrow 0$ such that  
\begin{equation}
\label{weak_error}
\upsilon_\varepsilon^{-1} \left(\EE F(L_{T})-\EE F(L_{T}^{\varepsilon}) \right) \rightarrow C_{F}\quad \mbox{ as } \varepsilon\searrow 0. \tag{$\mbox{WE}_{\upsilon_\varepsilon}$}
\end{equation}
We recall, in what follows, an important moment property of L\'evy processes. For this, we introduce before the below definition. 
\begin{definition}
A function $f: \RR^d \mapsto [0, \infty)$ is said to be  submultiplicative if there exists a positive constant $c$ such that $f(x+y) \leq c f(x) f(y)$ for $x, y \in \RR^d$. The product of two submultiplicative functions is also  submultiplicative.
\end{definition}
\begin{theorem}[Sato \cite{Sato}, Theorem 25.3]
\label{TH-Sato}
Let $f$ be a submultiplicative, locally bounded, measurable function on $\RR^d$, and let $(L_{t})_{t \geq 0}$ be a L\'evy process in $\RR^d$ with L\'evy measure $\nu$.
Then, $\EE f(L_t)$ is finite for every $t>0$ if and only if $\int_{|z| \geq 1} f(z) \nu(dz) < +\infty$. 
\end{theorem}
\section{Statistical Romberg method and Lévy process}
\label{SRLevy}
In this section, we establish two central limit theorems of Lindeberg-Feller type, for the inferred errors associated to MC and SR algorithms, in terms of the cut-off $\varepsilon$. Similar results are derived for the setting of an exponential Lévy model. We also provide a complexity analysis for both algorithms.
\subsection{Central limit theorem for the MC method}
\begin{theorem}
\label{clt:MC}
 Let $F:\RR^d\rightarrow \RR$ be a continuous function  satisfying  assumption \eqref{weak_error}. If $\sup_{0 < \varepsilon \leq 1}  \EE\left[ F^{2a} (L_{T}^{\varepsilon})\right] < +\infty$  for  $a>1$, then 
for $N= \upsilon_\varepsilon^{-2} $ we have
\begin{equation}
 \upsilon_\varepsilon^{-1} \left(\frac{1}{N} \sum_{i=1}^{N} F(L_{T,i}^{\varepsilon}) - \EE F(L_T) \right) \overset{\mathcal{L}}{\longrightarrow} \mathcal{N}(C_F,\V(F(L_T)))\quad\mbox{ as } \varepsilon \searrow 0.
\end{equation}
\label{CLT Monte Carlo}
\end{theorem}
\begin{proof}
At first, we write the total error as follows
\begin{equation*}
 \frac{1}{N} \sum_{i=1}^{N} F(L_{T,i}^{\varepsilon}) -\EE F(L_T)=\frac{1}{N} \sum_{i=1}^{N} F(L_{T,i}^{\varepsilon}) - \EE F(L_{T}^{\varepsilon}) + \left( \EE F(L_T^{\varepsilon}) - \EE F(L_T)\right).
\end{equation*}
Assumption \eqref{weak_error} ensures that $\lim_{\varepsilon \rightarrow 0}\upsilon_\varepsilon^{-1}  \EE\left( F(L_T^{\varepsilon})- F(L_T) \right)=C_{F}$. Concerning the first term on the right hand side of the above relation, as  $N$ depends on $\varepsilon$ we plan to apply the Lindeberg-Feller central limit theorem (see Theorem \ref{CLT Lindeberg Feller}). 
In order to do that, we set
$X_{i,\varepsilon}:=\frac{\upsilon_\varepsilon^{-1}}{N} \left(F(L_{T,i}^{\varepsilon})-\EE F(L_{T}^{\varepsilon}) \right)$ and we check assumptions ${\it A1}$ and ${\it A3}$ of  Theorem \ref{CLT Lindeberg Feller}. 
Thus, the proof is divided into two steps.
\\{{\bf {Step 1.}}\,} For assumption ${\it A1}$, it is straightforward that $\sum_{i=1}^{N} \EE(X_{i,\varepsilon}^2)=\V(F(L_{T}^{\varepsilon}))$. Then, by the almost sure convergence of $L_T^{\varepsilon}$ toward $L_T$, the continuity of  function $F$  and the uniform integrability condition given by $\sup_{0 < \varepsilon \leq 1}  \EE\left[ F^{2a} (L_{T}^{\varepsilon})\right] < +\infty$, we obtain
\begin{equation}
\lim\limits_{\varepsilon \to 0} \sum_{i=1}^{N} \EE(X_{i,\varepsilon}^2) =\lim\limits_{\varepsilon \to 0} \V(F(L_T^{\varepsilon}))= \V(F(L_T)).
\label{uniform integrability satisfied}
\end{equation}
{{\bf {Step 2.}}\,}
Concerning the Lyapunov condition ${\it{A3}}$, for $1< \tilde{a}< a$, we have 
\begin{equation*}
\sum_{i=1}^{N} \EE \left[ |X_{i,\varepsilon}|^{2\tilde{a}} \right] =\upsilon_\varepsilon^{2(\tilde{a}-1)} \EE \left| F(L_{T}^{\varepsilon})-\EE F(L_{T}^{\varepsilon}) \right|^{2\tilde{a}}.
\end{equation*}
Once again by the same arguments used in the previous step we prove the convergence of $\EE \left| F(L_{T}^{\varepsilon})-\EE F(L_{T}^{\varepsilon}) \right|^{2\tilde{a}}$ toward $\EE \left| F(L_{T})-\EE F(L_{T}) \right|^{2\tilde{a}}$ as $\varepsilon$ tends to zero.
Since $ \upsilon_\varepsilon^{2(\tilde{a}-1)} \underset{\varepsilon \rightarrow 0} {\longrightarrow} 0$, we obtain 
\begin{equation}
 \lim\limits_{\varepsilon \to 0} \sum_{i=1}^{N} \EE \left[ |X_{i,\varepsilon}|^{2\tilde{a}} \right] =0.
\label{Lyapunov condition satisfied}
\end{equation}
By (\ref{uniform integrability satisfied}) and (\ref{Lyapunov condition satisfied}), we obtain thanks to Theorem \ref{CLT Lindeberg Feller} the desired convergence in law.
\end{proof}
In the corollary below, we will treat the special case  where $F(x)=f(e^{x_1},\cdots,e^{x_d})$ for all $x=(x_1,\cdots,x_d)\in \RR^d$
and $f : \RR_+^d\rightarrow \RR$ is a Lipschitz continuous function.  In finance this model is well known as an exponential L\'evy model.
\begin{corollary}
\label{Cor-CLT-MC}
Assume that $\int_{|z|>1}e^{2a|z|}\nu(dz)$ is finite for $a>1$. Then,  in the setting of an exponential L\'evy model there is $C>0$ such that  
$\left|\EE F(L_{T})-\EE F(L_{T}^{\varepsilon})\right| \leq  C \sigma(\varepsilon)$. 
Moreover, if we choose $N= \sigma^{-2+\eta}(\varepsilon)$, with  $0<\eta<2$, then
\begin{equation}
  \sigma^{-1+\eta/2}(\varepsilon)\left(\frac{1}{N} \sum_{i=1}^{N} F(L_{T,i}^{\varepsilon}) - \EE F(L_T) \right) \overset{\mathcal{L}}{\longrightarrow} \mathcal{N}(0,\V(F(L_T)))\quad\mbox{ as } \varepsilon \searrow 0.
\end{equation}
\end{corollary}
\begin{proof}
We denote by $e^x$ the  exponential function  element-wise of the vector $x=(x_1,\cdots,x_d)\in \RR^d$,
$e^x=(e^{x_1},\cdots,e^{x_d})$. Let $C_f$ denote the Lipschitz constant of function $f$, since $L^1_T$ and 
$(L_T- L^1_T,L_T^{\varepsilon}- L^1_T)$ are independent we obtain by standard calculations  
\begin{eqnarray*}
\left|\EE F(L_T)-\EE F(L_T^{\varepsilon})\right|
&\leq& C_f \EE e^{\left|L_{T}^1 \right|} \EE  \left|L_T- L_T^{\varepsilon} \right|(e^{\left| L_T-L_{T}^1\right|}+ e^{\left|L_T^{\varepsilon}-L_{T}^1\right|})
\\&\leq& C_f \sigma(\varepsilon)\EE e^{\left|L_{T}^1 \right|}  \left(\bigl\|e^{ |L_T-L_{T}^1|}\bigr\|_2+\bigl\| e^{|L_T^{\varepsilon}-L_{T}^1|}\bigr\|_2\right).
\end{eqnarray*}
Now, on the one hand thanks to Theorem \ref{TH-Sato}, the assumption $\int_{|z|>1}e^{2a|z|}\nu(dz)<+\infty$  ensures  the finiteness of $\EE e^{\left|L_{T}^1 \right|}$. On the other hand by virtue of Lemmas 25.6 and 25.7 in Sato \cite{Sato} we have the boundedness of $\bigl\|e^{ |L_T-L_{T}^1|}\bigr\|_2$. Concerning the term $\bigl\| e^{|L_T^{\varepsilon}-L_{T}^1|}\bigr\|_2$,  we have  
$e^{|x|}\leq  \prod_{j=1}^{d} (e^{x_j}+e^{-x_j})$, this last upper bound can be written 
as a sum of finite number of exponential functions evaluated at points which are a linear combination of the components of the vector $x$. Therefore
there exists a family of $\RR^d$-valued vectors, $(b_j)_{1\leq j \leq 2^d}$ such that 
 $$
\bigl\| e^{|L_T^{\varepsilon}-L_{T}^1|}\bigr\|_2^2\leq 
 \sum_{j=1}^{2^d} \exp \left\{ T \int_{\varepsilon\leq |x|\leq 1} (e^{b_j.x} -1 - b_j.x) \nu(dx) \right\}. 
$$
Note that the finiteness of the above upper bound is once again  ensured  by Lemmas 25.6 and 25.7 in Sato \cite{Sato}. 
Since its limit exists  we deduce that $\sup_ {0 < \varepsilon \leq 1}\bigl\| e^{|L_T^{\varepsilon}-L_{T}^1|}\bigr\|_2$
is finite. Now, thanks to the linear growth of $f$ and using the same arguments as above we check in the same manner the property $\sup_{0 < \varepsilon \leq 1}  \EE\left[ F^{2a} (L_{T}^{\varepsilon})\right] < +\infty$. Hence, if we choose $\upsilon_\varepsilon=\sigma^{1-\eta/2}(\varepsilon)$ then 
Theorem \ref{CLT Monte Carlo} applies and this completes the proof.
\end{proof}

\subsection{Central limit theorem for the SR method}
\label{CLT-RS}
We use the SR method to approximate $\EE [F(L_T)]$ by
\begin{equation*}
 Q_{\varepsilon}= \frac{1}{N_1} \sum_{i=1}^{N_1} F(L_{T,i}^{\varepsilon^{\beta}})  + \frac{1}{N_{2}} \sum_{i=1}^{N_2} \left(F(L_{T,i}^{\varepsilon}) - F(L_{T,i}^{\varepsilon^{\beta}})\right)
\end{equation*}
\begin{theorem}
\label{CLT Statistical Romberg}
 Let $F:\RR^d\rightarrow \RR$ be a $\mathscr {C}^1$ function  satisfying assumption \eqref{weak_error}  and such that ${\sup_{0 < \varepsilon \leq 1} } \EE F^{2a} (L_{T}^{\varepsilon})$ and ${\sup_{0 < \varepsilon \leq 1} }\EE \left| \sigma^{-1}(\varepsilon) (F(L_{T}^{\varepsilon}) - F(L_{T}))\right|^{2a}$ are finite, for  $a>1$. Moreover, assume that
\begin{itemize}
 \item [$\it H1.$] Condition (\ref{condition_convergence1}) in Theorem \ref{Cohen_Rosinski} holds and there exists a definite positive  matrix $\Sigma$  such that $\lim\limits_{\varepsilon \to 0}\sigma^{-2}(\varepsilon)\Sigma_{\varepsilon}=\Sigma$.
 \item [$\it H2.$] For $0<\beta <1$, we have $\lim\limits_{\varepsilon \to 0}\sigma(\varepsilon)\sigma^{-1}(\varepsilon^\beta)=0$ and  
 $\lim\limits_{\varepsilon \to 0}\veps\sigma^{-1}(\varepsilon^\beta)=0$.
\end{itemize}
If we choose $N_1=\upsilon_\varepsilon^{-2}$ and $N_2= \veps^{-2}\sigma^{2}(\varepsilon^\beta) $, then  
\begin{equation*}
 \veps^{-1}\left(Q_{\varepsilon} -\EE F(L_T) \right) \xrightarrow{\mathcal{L}} \mathcal{N} \Bigl(C_{F}, \V(F(L_T))+ 
 T\EE(\nabla F(L_T). \Sigma \nabla F(L_T))\Bigr), \quad \mbox{as  } \varepsilon \searrow 0.
\end{equation*}
\end{theorem}
\begin{proof}
At first we  write the  total error as
$
 Q_{\varepsilon} -\EE F(L_T) = Q_{\varepsilon}^{1}+ Q_{\varepsilon}^{2}+ \EE F(L_T) - \EE F(L_{T}^{\varepsilon}),
$
with 
$$
Q_{\varepsilon}^{1}= \frac{1}{N_{1}} \sum_{i=1}^{N_1} F(L_{T,i}^{\varepsilon^{\beta}}) - \EE F(L_{T}^{\varepsilon^{\beta}}) \;\mbox{ and } \;
Q_{\varepsilon}^{2}= \frac{1}{N_{2}} \sum_{i=1}^{N_2} F(L_{T,i}^{\varepsilon}) -F(L_{T,i}^{\varepsilon^{\beta}}) - \EE \left[ F(L_{T}^{\varepsilon}) -F(L_{T}^{\varepsilon^{\beta}}) \right]. 
$$
 So, assumption  \eqref{weak_error} yields the convergence of $\upsilon_\varepsilon^{-1} \left(\EE F(L_{T})-\EE F(L_{T}^{\varepsilon}) \right) $ toward  $C_{F}$ as $\varepsilon$ goes to zero
and following step by step the proof of Theorem \ref{CLT Monte Carlo} the convergence law of $\upsilon_\varepsilon^{-1} Q_{\varepsilon}^{1}$  to the normal distribution
$
\mathcal{N}(0,\V(F(L_T)))
$
is easily obtained.
Concerning the term  $Q_{\varepsilon}^{2}$, we plan to use Theorem \ref{CLT Lindeberg Feller} and we set 
$
 X_{i,\varepsilon}:= \frac{\veps^{-1}} {N_2} \left( F(L_{T,i}^{\varepsilon}) -F(L_{T,i}^{\varepsilon^{\beta}}) - \left( \EE F(L_{T}^{\varepsilon}) - \EE F(L_{T}^{\varepsilon^{\beta}}) \right) \right).
$
In the following two steps, we will check   assumptions ${\it A1}$ and ${\it A3}$ of  Theorem \ref{CLT Lindeberg Feller}.
\vspace{0.25cm}

\noindent {{\bf {Step 1.}}\,} It is straightforward that $\sum_{i=1}^{N_2} \EE(X_{i,\varepsilon}^2)=\sigma^{-2}(\varepsilon^\beta) \V(F(L_{T}^{\varepsilon}) -F(L_{T}^{\varepsilon^{\beta}}))$.
Now applying Taylor-Young's expansion to the real valued $\mathcal{C}^{1}$  function $F$ we get
\begin{equation*}
 F(L_{T}^{\varepsilon}) - F(L_{T}^{\varepsilon^{\beta}})= \nabla F(L_{T}^{\varepsilon^{\beta}}).(L_{T}^{\varepsilon}-L_{T}^{\varepsilon^{\beta}}) + 
 (L_{T}^{\varepsilon}-L_{T}^{\varepsilon^{\beta}}).\epsilon( L_{T}^{\varepsilon}-L_T^{\varepsilon^{\beta}}),
\end{equation*}
where  $\epsilon(L_{T}^{\varepsilon}-L_T^{\varepsilon^{\beta}}) \overset{a.s.}{\longrightarrow} 0$ as $\varepsilon \rightarrow 0$. Now, by applying twice  Theorem \ref{Cohen_Rosinski} to $L_{T}^{\varepsilon}-L_{T}$ and $L_{T}-L_{T}^{\varepsilon^{\beta}}$ and thanks to assumption $\it H2$ we obtain 
$  \sigma^{-1}(\varepsilon^\beta)\bigl(L_{T}^{\varepsilon}-L_{T}^{\varepsilon^{\beta}} \bigr)\overset{\mathcal{L}} {\underset{\varepsilon \rightarrow 0} {\longrightarrow}} \Sigma^{1/2} W_T$.
Since $L_{T}^{\varepsilon^{\beta}}$ is independent from $L_{T}^{\varepsilon}-L_{T}^{\varepsilon^{\beta}}$ and $\nabla F(L_T^{\varepsilon^{\beta}}) \overset{a.s.}{\underset{\varepsilon \rightarrow 0}{\longrightarrow}} \nabla F(L_T)$
, we obtain
\begin{equation}
\label{convergence}
 \sigma^{-1}(\varepsilon^\beta) \left( F(L_{T}^{\varepsilon})-F(L_{T}^{\varepsilon^{\beta}}) \right)  \overset{\mathcal{L}} {\underset{\varepsilon \rightarrow 0}{\longrightarrow}} \nabla F(L_T). \Sigma^{1/2} W_T
\end{equation}
For the second term, using the tightness of $  \sigma^{-1}(\varepsilon^\beta)\bigl(L_{T}^{\varepsilon}-L_{T}^{\varepsilon^{\beta}} \bigr)$  we deduce that $ \sigma^{-1}(\varepsilon^\beta)\bigl(L_{T}^{\varepsilon}-L_{T}^{\varepsilon^{\beta}} \bigr)\epsilon(L_{T}^{\varepsilon}-L_T^{\varepsilon^{\beta}})  \overset{a.s.}{\underset{\varepsilon \rightarrow 0}{\longrightarrow}} 0.$ Thanks to the inequality 
$|x+y|^{2a}\leq 2^{2a-1}(|x|^{2a}+ |y|^{2a})$, for any $x,y\in \RR$,
${\sup_{0 < \varepsilon \leq 1} }\EE \left| \sigma^{-1}(\varepsilon) (F(L_{T}^{\varepsilon}) - F(L_{T}))\right|^{2a}<+\infty$ and $\lim\limits_{\varepsilon \to 0}\sigma(\varepsilon)\sigma^{-1}(\varepsilon^\beta)=0$ we deduce the uniform integrability of $
\sigma^{-2}(\varepsilon^\beta)|F(L_{T}^{\varepsilon})-F(L_{T}^{\varepsilon^{\beta}})|^2$.
Therefore, we obtain the first condition
$$
 \lim\limits_{\varepsilon \to 0} \sum_{i=1}^{N_2} \EE(X_{i,\varepsilon})^2= \V(\nabla F(L_T).\Sigma^{1/2} W_T)=T\EE(\nabla F(L_T).\Sigma \nabla F(L_T)).
$$
{{\bf {Step 2.}}\,} For the Lyapunov condition,  let $1<a'<a$, we get by standard evaluations
\begin{equation*}
\sum_{i=1}^{N_2} \EE |X_{i,\varepsilon}|^{2a'}  \leq  2^{2a'} \veps^{2(a'-1)} 
\sigma^{-2(a'-1)}(\varepsilon^\beta)  \EE \left|\sigma^{-1}(\varepsilon^\beta) ( F(L_{T}^{\varepsilon}) -F(L_{T}^{\varepsilon^{\beta}})) \right|^{2a'}.
\end{equation*}
Once again we use the convergence in distribution given by relation (\ref{convergence}) and the uniform integrability property ${\sup_{0 < \varepsilon \leq 1} }\EE \left| \sigma^{-1}(\varepsilon^\beta) (F(L_{T}^{\varepsilon})- F(L_{T}^{\varepsilon^{\beta}})) \right|^{2a}<+\infty$ to deduce the convergence of
$\EE \left|\sigma^{-1}(\varepsilon^\beta) ( F(L_{T}^{\varepsilon}) -F(L_{T}^{\varepsilon^{\beta}})) \right|^{2a'}$ toward $\EE \left| \nabla F(L_T). \Sigma^{1/2} W_T \right|^{2a'}$.
Finally, since $\lim\limits_{\varepsilon \to 0}\veps\sigma^{-1}(\varepsilon^\beta)=0$, we conclude that $\lim\limits_{\varepsilon \to 0} \sum_{i=1}^{N_2} \EE |X_{i,\varepsilon}|^{2a'}= 0$ with $a'>1$.
This gives the asymptotic normality of $Q_{\varepsilon}^2$ and completes the proof.
\end{proof}
Now, we get back to the exponential Lévy model setting introduced  before Corollary \ref{Cor-CLT-MC} 
where $F(x)=f(e^{x_1},\cdots,e^{x_d})$ for a given $\mathscr C^1$ Lipschitz continuous function $f$. 
Our aim is to deduce in this setting a central limit theorem for SR method. 
\begin{corollary}
\label{Cor-CLT-RS}
Assume that $\int_{|z|>1}e^{2a|z|}\nu(dz)$ is finite for $a>1$. In the setting of an exponential L\'evy model there is $C>0$ such that  
$\left|\EE F(L_{T})-\EE F(L_{T}^{\varepsilon})\right| \leq  C \sigma(\varepsilon)$. Moreover, assume that for $0<\beta <1$ there exists 
$0<\eta<2$ such that $\lim\limits_{\varepsilon \to 0}\sigma^{1-\eta/2}(\varepsilon)\sigma^{-1}(\varepsilon^\beta)=0$, 
$\sigma(\varepsilon)>\varepsilon$ for all  $0<\varepsilon<1$  and condition $\it H1$ of Theorem \ref{CLT Statistical Romberg} is satisfied.
Then, if we choose $N_1= \sigma^{-2+\eta}(\varepsilon)$ and $N_2= \sigma^{-2+\eta}(\varepsilon)\sigma^{-1}(\varepsilon^\beta)$ we obtain
\begin{equation*}
  \sigma^{-1+\eta/2}\left(Q_{\varepsilon} -\EE F(L_T) \right) \xrightarrow{\mathcal{L}} \mathcal{N} \Bigl(0, \V(F(L_T))+ 
 T\EE(\nabla F(L_T). \Sigma \nabla F(L_T))\Bigr), \quad \mbox{as  } \varepsilon \searrow 0.
\end{equation*}
\end{corollary}
\begin{proof}
According to Theorem \ref{CLT Statistical Romberg} and Corollary \ref{Cor-CLT-MC} we  only need to check that assumption
${\sup_{0 < \varepsilon \leq 1} }\EE \left| \sigma^{-1}(\varepsilon) (F(L_{T}^{\varepsilon}) - F(L_{T}))\right|^{2a}<\infty$ is satisfied.
Since $f$ is Lipschitz it is sufficient to find an upper bound for $\EE\left|e^{ L_{T}^{\varepsilon}}-e^{L_{T}} \right|^{2a}$. To do so, we use 
the independence of $L^1_T$ and the couple $(L_T-L_T^1,L_T^{\varepsilon}-L_T^1)$ and Cauchy-Schwartz's inequality  to get  
$$
\EE\left|e^{ L_{T}^{\varepsilon}}-e^{L_{T}} \right|^{2a}
  \leq  \EE e^{ 2a|L_T^1|} \bigl\||L_T- L_T^{\varepsilon}|^{2a} \bigr\|_{2} \left(\bigl\|e^{ 2a|L_T-L_T^1|}\bigr\|_{2}+\bigl\| e^{2a|L_T^{\varepsilon}-L_T^1|}\bigr\|_{2}\right).
 $$
By the same arguments given in the proof of Corollary \ref{Cor-CLT-MC} we have the finiteness of $\EE e^{ 2a|L_T^1|}$, $\bigl\|e^{ 2a|L_T-L_T^1|}\bigr\|_{2}$ and $\sup_{0<\varepsilon \leq 1} \bigl\| e^{2a|L_T^{\varepsilon}-L_T^1|}\bigr\|_{2}$. Combining all these results together with
assumption \eqref{bounds_error_dia} we  deduce the existence of a constant $C>0$ not depending on $\varepsilon$ such that
 $$
\EE \left| \sigma^{-1}(\varepsilon) (F(L_{T}^{\varepsilon}) - F(L_{T}))\right|^{2a}  \leq C\sigma^{-2a}(\varepsilon) \sigma_0^{2a}(\varepsilon). 
$$
This completes the proof since $\sigma_{0}(\varepsilon)=\sigma(\varepsilon)$, for  $0<\varepsilon<1$.
\end{proof}
\subsection{Complexity Analysis}
\label{complexity}
Thanks to the above limit results we are able now to provide a complexity analysis for both MC and SR algorithm.
To keep things simple, we consider the particular case $d=1$, $v_{\varepsilon}=\sigma(\varepsilon)$ and we assume that the measure $\nu$ has a density of the form $L(x)/|x|^{Y+1}$ for a small $x$, where $L(x)$ is a slowly varying as $x\rightarrow 0$ and $Y\in(0,2)$. Observe that the positive (resp. negative ) part  of the  approximation $(L^{\varepsilon}_t)_{0\leq t\leq T}$ is essentially a compound Poisson process with intensity $\nu( [\varepsilon, +\infty ))$ (resp. $\nu( (-\infty,-\varepsilon])$). Then, the cost necessary of a single simulation is random, with expectation of order $\mathcal K(\varepsilon)=\nu(|x|\geq \varepsilon).$
Hence, according to Theorem \ref{clt:MC} the time complexity  of the  MC  method  necessary to achieve a  total error of order $\sigma(\varepsilon)$ is random with expectation of order
\begin{equation*}
 C_{MC} =  \mathcal K(\varepsilon) N=\mathcal K(\varepsilon) \sigma^{-2}(\varepsilon).
\end{equation*}
In the same way, thanks to Theorem \ref{CLT Statistical Romberg}  the time complexity  of the SR method necessary to achieve a  total error of order $\sigma(\varepsilon)$  is random with expectation of order
\begin{equation*}
 C_{SR} =  \mathcal K(\varepsilon^{\beta}) N_1+\mathcal K(\varepsilon) N_2
 =\left(\mathcal K(\varepsilon^{\beta}) +\mathcal K(\varepsilon)\sigma^{2}(\varepsilon^{\beta})\right)\sigma^{-2}(\varepsilon).
\end{equation*}
By Karamata's theorem (see e.g. Bingham, Goldie and Teugels \cite{Bingham} or Feller \cite{Feller} )
$$ 
\sigma^2(\varepsilon)=\int_{-\varepsilon}^{\varepsilon}|x|^{1-Y}L(x)dx\sim 
\frac{L(\varepsilon)+L(-\varepsilon)}{2-Y}\varepsilon^{2-Y}.
$$
Similarly we have 
$$
\mathcal K(\varepsilon)\sim \frac{L(\varepsilon)+L(-\varepsilon)}{Y}\varepsilon^{-Y}.
$$
Consequently, we compute the time complexity ratio given by 
$$\frac{C_{SR}}{C_{MC}}=\frac{L(\varepsilon^{\beta}) +L(-\varepsilon^{\beta})}{L(\varepsilon)+L(-\varepsilon^{\beta})}\varepsilon^{Y(1-\beta)} +\frac{L(\varepsilon^{\beta}) +L(-\varepsilon^{\beta})}{2-Y}\varepsilon^{\beta(2-Y)}.$$
If $L(\varepsilon)$  is constant in the neighborhood of zero, like for the CGMY model (see relation \eqref{eq:CGMY}), then we easily get
$$\frac{C_{SR}}{C_{MC}}= O\left( \varepsilon^{Y(1-\beta)}+\varepsilon^{\beta(2-Y)}\right).$$
Optimizing the order of this last quantity yields $\beta=Y/2$  which leads us to a gain of a complexity of order $\varepsilon^{Y(Y/2-1)}$ that asymptotically increases as soon as  $\varepsilon$ becomes small.  
\section{Importance Sampling and Statistical Romberg  method}
\label{sec:IS}
Let $\{L_t ; t \geq 0 \}$ be a L\'evy process in $\RR^{d}$ under the probability $\PP$ with generating triplet $(\gamma, 0, \nu)$.
We define the set 
\begin{equation}
\Theta_1 := \bigl\{ \theta \in \RR^{d} : \EE[e^{\theta.L_t}] < +\infty \bigr\} = \bigl\{ \theta \in \RR^{d} : \int_{|x|>1} e^{\theta.x} \nu(dx) < \infty \bigr\},
\end{equation}
where the second equality holds by Theorem \ref{TH-Sato}. Thanks to the convexity of the exponential function it is straightforward that the set $\Theta_1$ is convex. 
In view to use importance sampling routine, based on exponential tilting, we define the family of $\left\{\PP_{\theta}, \theta \in \Theta_1\right\}$, as all the equivalent probability measures with respect to $\PP$ such that
\begin{equation*}
  \frac{ d\PP_{\theta}}{d\PP}\bigr|_{\mathcal{F}_t} = \frac{e^{\theta. L_{t}}}{\EE [e^{\theta. L_{t}}]}= e^{\theta. L_{t}- t \kappa(\theta)}
\end{equation*}
whee  $\kappa$ denotes the cumulant generating function given by $\kappa (\theta)=\ln \EE \left[ e^{\theta. L_1} \right]$. 
Under $\PP_{\theta}$, the stochastic process $\{L_t ; t \geq 0 \}$ is still a L\'evy process with the exponential tilted triplet $(\gamma_{\theta}, 0, \nu_{\theta})$ where 
$\gamma_{\theta}=\gamma + \int_{|x| \leq 1} x(\nu_{\theta} - \nu) (dx)$ and $\nu_{\theta} (dx) = e^{\theta.x} \nu(dx)$ (see e.g. Cont and Tankov \cite{ContTankov}). 
Hence, we obtain
$
\EE \left[ F(L_T) \right] =\EE_{\theta} \left[ F(L_{T}) e^{-\theta. L_{T} + T \kappa(\theta)}  \right]. 
$
If we introduce the L\'evy process $\{L_t^{\theta} ; t \geq 0 \}$ with generating triplet $(\gamma_{\theta}, 0, \nu_{\theta})$ under $\PP$, then
the random variable $L_T$ under $\PP_{\theta}$ has the same law as $L_T^{\theta}$ under $\PP$ and we get
\begin{equation*}
 \EE \left[ F(L_T) \right]= \EE \left[ F(L_{T}^\theta) e^{-\theta. L_{T}^\theta + T \kappa(\theta)}  \right].
\end{equation*}
Further, one can use this importance sampling twice in the  SR algorithm  with considering $\theta_1$ and $\theta_2$ in $\RR^d$ and  approximate $\EE [F(L_T)]$ by
\begin{equation*}
 \frac{1}{N_1} \sum_{k=1}^{N_1} F(L_{T,k}^{\varepsilon^{\beta}, \theta_1}) e^{-\theta_1.L_{T,k}^{\varepsilon^{\beta},\theta_1} + T \kappa(\theta_1)} + \frac{1}{N_{2}} \sum_{k=1}^{N_2} (F(L_{T,k}^{\varepsilon, \theta_2})  - F(L_{T,k}^{\varepsilon^{\beta}, \theta_2})) e^{-\theta_2.L_{T,k}^{\varepsilon,\theta_2 } + T \kappa(\theta_2)}.
\end{equation*}
Miming the proof of Theorem \ref{CLT Statistical Romberg} we establish a central limit theorem with limit variance 
$
 \V (F(L_{T}^{\theta_1}) e^{-\theta_1 L_{T}^{\theta_1} + T \kappa(\theta_1)})+ T\EE ( (\nabla F(L_{T}^{\theta_2}).\Sigma \nabla F(L_{T}^{\theta_2}) ) e^{-2\theta_2 L_{T}^{\theta_2} + 2T \kappa(\theta_2)} ).
$ 
Since $L_{T}^{\theta_1}$ (resp. $L_{T}^{\theta_2}$) under $\PP$ has the same law as $ L_{T}$ under $\PP_{\theta_1}$ (resp. $\PP_{\theta_2}$) we rewrite this  variance using once again the Esscher transform  as 
$$
 \EE \left[ F^{2}(L_{T}) e^{-\theta_1.L_{T} + T \kappa(\theta_1)} \right]  - \left[\EE  F(L_{T}) \right]^2+ 
 T\EE \left[ (\nabla F(L_{T}).\Sigma \nabla F(L_{T}) ) e^{-\theta_2 L_{T} + T \kappa(\theta_2)} \right].
$$
Hence, let us introduce for $i\in\{1,2\}$,
\begin{equation}
\label{var_theta}
 v_i(\theta):=\EE \left[ F_i(L_{T})e^{-\theta L_{T} + T \kappa(\theta)} \right],  \text{  with  } 
  F_1 \equiv F^2 \text{ and } F_2 \equiv \nabla F.\Sigma \nabla F .
\end{equation} 
Our aim now is to minimize separately these two quantities. To do so, for $i\in\{1,2\}$, we introduce a first 
 set
 $$
 \Theta_{i,2}:=\Theta_1\cap \left\{ \theta \in \RR^d : 
 \EE \left[ F_i(L_{T} ) e^{-\theta . L_{T} } \right] < +\infty \right\}
$$
to ensure the existence of $v_i(\theta)$  and a second set
$$
 \Theta_{i,3}:=\Theta_{i,2}\cap \left\{ \theta \in \RR^d : 
 \EE \left[  |L_{T}|^2 F_i(L_{T}) e^{-\theta . L_{T} } \right] < +\infty \right\}
$$
to make sens for the first and second derivatives of $v_i(\theta)$.
For $i\in\{1,2\}$, if we assume that $Leb(\Theta_{i,3}) >0$, then the convexity of sets $\Theta_{i,2}$ and $\Theta_{i,3}$ can be proved in a similar manner to the proof of Lemma 2.2 in \cite{Kawai}. Moreover, we prove the convexity of $v_i$, $i\in\{1,2\}$.
\begin{proposition}
\label{convexity_v}
Let $i\in\{1,2\}$.
Assume $\PP(F_i(L_T)\neq 0) >0.$ Then,  $\theta \mapsto v_i(\theta)$  
is a  $\mathscr C^2$ strictly convex function
on $\Theta_{i,3}$ and  $\nabla v_i(\theta)=\EE \left[ H_i(\theta, L_T )\right]$  where
\begin{equation}
\label{gradiant_V}
 H_i(\theta, L_T)= (T \nabla \kappa(\theta) -L_T) F_i(L_T) \exp(-\theta.L_T + T \kappa(\theta)).
\end{equation}
\end{proposition}
\paragraph{Proof.} For a fixed $i\in \{1,2\}$,
 the function $\theta \mapsto F_i(L_{T})  e^{-\theta L_{T} + T \kappa(\theta)}$ is  almost surely differentiable on $ \Theta_1$ with a first derivative equal to $ H_i(\theta, L_T)$. Further,  
according to the properties of the moment generating function,  the function $\theta \mapsto v_i(\theta)$ is finite for $\theta \in \Theta_{i,2}$ and is differentiable  
with  $\nabla v_i(\theta)=\EE \left[ H_i(\theta, L_T )\right]$ provided that $\EE \left[|H_i(\theta, L_T)| \right]$ is finite.
Using Hölder's inequality, this last condition is satisfied as soon as $\theta \in \Theta_{i,3}$.
 In the same way, we prove that $v_i$ is of class $\mathscr C^{2} $ on  $\Theta_{i,3}$ and we get for all $u \in \RR^d \setminus \{ 0\}$, 
 $$
u.\textnormal{Hess}(v_i(\theta)) u=\EE\left[ \left( u.\textnormal{Hess}(\kappa(\theta))u + \left(u. (T \nabla \kappa(\theta)-L_T)\right)^2 \right)  F_i(L_T)e^{-\theta.L_T+T\kappa(\theta)} \right].
$$
Note that $\textnormal{Hess}(\kappa(\theta))$ is nothing but the variance-covariance matrix of the random vector $L_T$ under the probability measure $\PP_\theta$ and it is clearly definite positive. Finally, since  $\PP(F_i(L_T) \neq 0) >0$, we conclude that $v_i$ is strictly convex on $ \Theta_{i,3}$.
$\hfill\square$ 

For $\varepsilon >0$, the same result holds for the approximated L\'evy process $(L_t^{\varepsilon})_{t\geq 0}$ by considering the associated 
sets $\Theta_{1}^{ \varepsilon}$, $\Theta_{i,2}^{ \varepsilon}$  and $\Theta_{i,3}^{\varepsilon}$ and 
functions $\kappa_{\varepsilon}$ and $v_{i,\varepsilon}$, $i\in\{1,2\}$, with the canonical filtration $(\mathcal F^{\varepsilon}_t)_{0\leq t\leq T}$ defined by $\mathcal F^{\varepsilon}_t=\sigma(L^{\varepsilon}_s,s\leq t)$. 

\begin{proposition}
\label{convexity_v_epsilon}
Let $i\in\{1,2\}$. Assume $\PP(F_i(L_T^{\varepsilon})\neq 0) >0$ then the function $v_{i,\varepsilon}(\theta)=\EE \left[F_i(L_{T}^{\varepsilon})  e^{-\theta L_{T}^{\varepsilon} + T \kappa_{\varepsilon}(\theta)} \right]$ is of class $\mathscr C^2$ and strictly convex on $\Theta_{i,3}^{\varepsilon}$  with $\nabla v_{i,\varepsilon}(\theta)=\EE \left[ H_i(\theta, L_T^{\varepsilon})\right]$.
\end{proposition}
Now, let us introduce for $i\in\{1,2\}$
\begin{equation}
 \label{optimal_theta}
 \theta^*_{i,\varepsilon}:=\argmin_{\theta\in \Theta_{i,3}^{\varepsilon} } v_{i,\varepsilon}
(\theta)\quad \mbox{ and } \quad
\theta^*_{i}:=\argmin_{\theta \in \Theta_{i,3}} v_i(\theta).
\end{equation}
Our aim now is to study for  $i\in\{1,2\}$ the convergence of 
$\theta^*_{i,\varepsilon}$ toward
$\theta^*_{i}$ 
 as ${\varepsilon}$ tends to zero. For $q > 1$, we define the set
\begin{equation}
\label{theta_q}
\Theta_{q}:= \left\{ \theta \in \RR^d : \int_{|x| > 1}  |x|^{2q} e^{- q \theta.x} \nu(dx) < +\infty  \right\}.
\end{equation}
\paragraph{Remark.}
\begin{enumerate}
 \item 
It is worth to note that for $ 0\leq q'\leq 2q$  and $\theta\in\Theta_{q}$ we have $ \int_{|x| > 1}  |x|^{q'} e^{- q \theta.x} \nu(dx) < +\infty$. We also have $\Theta_{q_2}\subset\Theta_{q_1}$ for all $q_1\leq q_2$. 
\item Further, for $i\in\{1,2\}$, if $ \mathbb{E}\left[F_i^{a}(L_{T})\right]$, $a>1$, is finite then by Hölder's inequality we easily get   $\Theta_{q}\subset \Theta_{i,3}$ for all $q\geq{a}/{a-1}$.
The same result holds for the approximated Lévy process. Indeed, for $\varepsilon >0$,  we have $\Theta_{q}\subset \Theta_{i,3}^{\varepsilon}$ provided that $\mathbb{E}\left[F_i^{a}(L^\varepsilon_{T})\right]<\infty$. 
\end{enumerate}

According the above remark, choosing $\theta\in\Theta_{q} $ with $q\geq{a}/{a-1}$ ensures that  $\theta$ will belong to the domain of 
convexity of both  $v_{i}$ and  $ v_{i,\varepsilon}$. On the other hand it also guarantees the finiteness of  the quantity $\int_{|x| > 1}  |x|^{q} e^{- q \theta.x} \nu(dx)$ which will be needed in each proof assuming condition $\theta\in\Theta_{q}$. 

In what follows,  let $ \mathring E$ denote the set of all interior points of a given set $E$. We have the following result.
\begin{theorem}
\label{th:convergence}
Let $i\in\{1,2\}$.  Suppose  that $x\mapsto F_i(x)$ is continuous, that is for the case $i=1$ the function $F$ is continuous and for $i=2$ the function $F$ is of class $\mathscr C^1$. Moreover, assume $\PP(F_i(L_T)\neq 0) >0$,  $\PP(F_i(L_T^{\varepsilon})\neq 0) >0$ for all $\varepsilon>0$ and there exists $ a > 1 $  such that $ \mathbb{E}\left[F_i^{a}(L_{T})\right]$ and $\sup_{\varepsilon>0}\mathbb{E}\left[F_i^{a}(L^\varepsilon_{T})\right]$ are finite.
Let $K$ be a compact set such that $K\subset \mathring\Theta_{q}$ with $q > \frac{a}{a-1}$ and assume that the sequence  $(\theta_{i,\varepsilon}^*)_{\varepsilon>0}\in K $. Then,  
$$\theta_{i,\varepsilon}^{*}{\longrightarrow}\theta_{i}^{*}\in K,\quad\mbox{ as }{\varepsilon \rightarrow 0}.$$
\end{theorem}
We prove  Theorem \ref{th:convergence} after the following technical lemma.
\begin{lemma}
\label{lem:ucv}
 Let $K$ be a compact subset of $\Theta_q$ with $q>1$, we have 
 $
\sup_{\theta\in \Theta_q}{\mathbb E}\bigl[|
 L^{\varepsilon}_{T}|^q 
  e^{- q\theta .L_T^{\varepsilon}}\bigr]
$
is uniformly bounded in $\varepsilon$.  
\end{lemma}
\begin{proof}
Let us consider the two independent Lévy processes $L^1$ and $\tilde L^{\varepsilon}:=L^{\varepsilon}-L^1$ and the submultiplicative function $g_{\theta}(x):=(|x|\vee 1)^qe^{- q \theta. x}$. There exists $c_q>0$ depending only on $q$ such that $g_{\theta}(x+y)\leq c_q g_{\theta}(x)g_{\theta}(y)$ for any $\theta\in \RR^d$ and
$$
{\mathbb E}\bigl[|
 L^{\varepsilon}_{T}|^q 
  e^{- q\theta .L_T^{\varepsilon}}\bigr]\leq c_q {\mathbb E}\bigl[ g_{\theta}(\tilde L_T^{\varepsilon})\bigr]{\mathbb E}\bigl[g_{\theta}(L_T^1)\bigr].
$$
Since the function $\theta \mapsto {\mathbb E} \bigl[g_{\theta}(L_T^1)\bigr] $ is continuous on $\Theta_q$ the second expectation on the right hand side is  uniformly bounded on $\theta\in K$. Concerning the first expectation, we start by establishing the uniform convergence of $\tilde \kappa_\varepsilon$ toward $\tilde \kappa$, where $\tilde \kappa_\varepsilon$ and $\tilde \kappa$ denote the cumulant generating functions of respectively 
$\tilde L^{\varepsilon}=L^{\varepsilon}-L^1$ and $\tilde L=L-L^1$. According to the L\'evy Kintchine decomposition, we have 
$
\tilde \kappa(\theta)-\tilde \kappa_\varepsilon(\theta)=\int_{|x|<\varepsilon}(e^{\theta.x}-1-\theta.x)\nu(dx)
$
and thanks to Taylor's expansion we get
\begin{equation}
\label{eq-UC-tilde-kapa}
|\tilde \kappa(\theta)-\tilde \kappa_\varepsilon(\theta)|\leq \frac{|\theta|^2}{2}e^{|\theta|} \sigma^2(\varepsilon).
\end{equation}
This ensures the uniform convergence of the family functions $(\tilde \kappa_\varepsilon)_{0<\varepsilon<1}$ on any compact set of $\RR^d$. Note 
that for all $x=(x_1,\cdots,x_d)\in \RR^d$ we have  
$(|x|\vee 1)^q\leq c e^{|x|}\leq c \prod_{j=1}^{d} (e^{x_j}+e^{-x_j})$ with some $c>0$ depending only on $q$. This last upper bound can be written 
as a sum of finite number of exponential functions evaluated at points which are a linear combination of the components of the vector $x$. Therefore
there exists a family of deterministic $\RR^d$-valued vectors, $(b_j)_{1\leq j \leq 2^d}$ such that 
 $$
{\mathbb E}\bigl[ g_{\theta}(\tilde L_T^{\varepsilon})\bigr]\leq c \sum_{j=1}^{2^d} {\mathbb E}\bigl[ e^{(b_j-q \theta).L_T^{\varepsilon} } \bigr]. 
$$
Each term in the above sum is nothing else $\exp(\tilde \kappa_{\varepsilon}(b_j-q \theta))$ which in turn converges 
to $\exp(\tilde \kappa(b_j-q \theta))$ as $\varepsilon$ tends to zero.  This gives us the desired claim. 
\end{proof}
\begin{proof}[Proof of Theorem \ref{th:convergence}] Let $i \in \{1,2\}$ and $(\varepsilon_n)_{n\in \NN}$ be a sequence decreasing to zero.  Note that  $(\theta_{i,\varepsilon_n}^{*})_{n\in \NN }$ is a $\RR^d$-bounded sequence.  So, we only need to prove 
that for any subsequence $(\theta_{i,\varepsilon_{n_k}}^{*})_{k\in\NN}$, if $\theta_{i,\varepsilon_{n_k}}^{*}\rightarrow \theta_{i,\infty}^{*} \in \RR^d$ then $\theta_{i,\infty}^{*}= \theta_i^{*}$.  According to Proposition \ref{convexity_v_epsilon} above we have 
$$
\nabla v_{i,\varepsilon_{n_k}}(\theta_{i,\varepsilon_{n_k}}^{*} )= {\mathbb E} \left[ (\theta_{i,\varepsilon_{n_k}}^{*} T-L^{\varepsilon_{n_k}}_{T}) 
 F_i(L_{T}^{\varepsilon_{n_k}})  e^{- \theta_{i,\varepsilon_{n_k}}^{*} .L_T^{\varepsilon_{n_k}} + 
T \kappa_{\varepsilon_{n_k}}(\theta_{i,\varepsilon_{n_k}}^{*})}
\right]=0.
$$
 Now,  let $\tilde a=\frac{aq}{a+q}$, it is easy to check  that $1<\tilde a<a$, so by applying Hölder's inequality
 we get
\begin{multline*}
  {\mathbb E}\left[ \bigl|
 (\theta_{i,\varepsilon_{n_k}}^{*} T-L^{\varepsilon_{n_k}}_{T}) 
 F_i(L_{T}^{\varepsilon_{n_k}})  e^{- \theta_{i,\varepsilon_{n_k}}^{*} .L_T^{\varepsilon_{n_k}} + 
T \kappa_{\varepsilon_{n_k}}(\theta_{i,\varepsilon_{n_k}}^{*})}
 \bigr|^{\tilde a}\right]\leq \\
  {\mathbb E}^{(a-\tilde a)/a}\left[ \bigl|
 (\theta_{i,\varepsilon_{n_k}}^{*} T-L^{\varepsilon_{n_k}}_{T}) 
  e^{- \theta_{i,\varepsilon_{n_k}}^{*} .L_T^{\varepsilon_{n_k}} + 
T \kappa_{\varepsilon_{n_k}}(\theta_{i,\varepsilon_{n_k}}^{*})}
 \bigr|^{\tilde a a/(a-\tilde a)}\right]
 \EE^{\tilde a / a} \left[ F_i^a(L_{T}^{\varepsilon_{n_k}})\right].
\end{multline*}
Note that $\sup_{\varepsilon>0}\mathbb{E}\left[F_i^{a}(L^\varepsilon_{T})\right]<\infty$. Hence, to get  the uniform integrability it is sufficient to prove that the first expectation on the right hand side of the above inequality is  uniformly bounded on $\varepsilon_{n_k}$ and $\theta_{i,\varepsilon_{n_k}}^{*}$. Indeed, using  
the  almost sure convergence  of $L^{\varepsilon}_{T}$ toward $L_{T}$ 
 and the continuity of function $F_i$, we easily get
 $$
\nabla v_{i}(\theta_{i,\infty}^{*} )= {\mathbb E} \left[ (\theta_{i,\infty}^{*} T-L_{T}) 
 F_i(L_{T})  e^{- \theta_{i,\infty}^{*}  .L_T + 
T \kappa(\theta_{i,\infty}^{*})}
\right]=0$$
and then we complete the proof using the uniqueness of the minimum ensured  by Proposition \ref{convexity_v}.
Consequently, noticing that $q=\tilde a a/(a-\tilde a)$, it remains now to prove  the uniform boundedness of the quantity 
$
{\mathbb E}\left[ \bigl|
 (\theta_{i,\varepsilon_{n_k}}^{*} T-L^{\varepsilon_{n_k}}_{T}) 
  e^{- \theta_{i,\varepsilon_{n_k}}^{*} .L_T^{\varepsilon_{n_k}} + 
T \kappa_{\varepsilon_{n_k}}(\theta_{i,\varepsilon_{n_k}}^{*})}
 \bigr|^q\right].
$
To do so, we establish first the uniform convergence of $\kappa_\varepsilon$ toward $\kappa$. According to the decomposition given by relation \eqref{error}, we have that 
$
\kappa(\theta)-\kappa_\varepsilon(\theta)=\int_{|x|<\varepsilon}(e^{\theta.x}-1-\theta.x)\nu(dx).
$
By Taylor's expansion we deduce
\begin{equation}
\label{eq-UC-kapa}
|\kappa(\theta)-\kappa_\varepsilon(\theta)|\leq \frac{|\theta|^2}{2}e^{|\theta|} \sigma^2(\varepsilon).
\end{equation}
Hence, the family functions $(\kappa_\varepsilon)_{0<\varepsilon<1}$ is equicontinuous on any compact subset of $\Theta_1$ and we deduce 
the convergence of $\kappa_{\varepsilon_{n_k}}(\theta_{i,\varepsilon_{n_k}}^{*})$ toward $\kappa(\theta_{i,\infty}^{*})$ when $k$ tends to infinity. Noticing that $-qK\subset \Theta_1$, we use once again the equicontinuity of  $(\kappa_\varepsilon)_{0<\varepsilon<1}$  on the compact set $-qK$  to get $\lim_{k\rightarrow\infty}\kappa_{\varepsilon_{n_k}}(- q\theta_{i,\varepsilon_{n_k}}^{*})=\kappa(-q\theta_{i,\infty}^{*})$ and then the problem is reduced to prove the uniform boundedness of $
{\mathbb E}\bigl[|
 L^{\varepsilon_{n_k}}_{T}|^q 
  e^{- q\theta_{i,\varepsilon_{n_k}}^{*} .L_T^{\varepsilon_{n_k}}}\bigr]
$
which is ensured by Lemma \ref{lem:ucv}.
\end{proof}
\section{The adaptive procedure}
\label{adaptativeprocedure}
\subsection{Stochastic algorithms}
The aim now is to construct family sequences converging almost surely to the optimal limits $\theta^*_{1,\varepsilon}$ and  $\theta^*_{2,\varepsilon}$ of the previous section. For this, let $ (L_{T,n})_{n \geq 1}$ (resp. $ (L_{T,n}^{\varepsilon})_{n \geq 1}$, $\varepsilon>0$), be i.i.d copies of the $\RR^d$-valued random variable $L_T$ (resp. $L_T^{\varepsilon}$). Let $K$ be a compact convex subset of $\Theta_1\subset \RR^d$ with $\{0\} \in K $. 
For fixed $i\in \{1,2\}$ and $\theta_{i,0} \in K$, we construct recursively the sequences of $\RR^d$-valued random variables 
$(\theta_{i,n})_{n \in \NN}$ and $(\theta_{i,\varepsilon,n})_{n \in \NN}$  defined  by the system
\begin{equation}
\label{sequence_theta}
\left\{
\begin{array}{lcl}
 \theta_{i,n+1} &=& \Pi_{K} \left[\theta_{i,n}-\gamma_{n+1} H_i(\theta_{i,n}, L_{T,n+1}) \right]\\
\theta_{i,\varepsilon,n+1} &=& \Pi_{K} \left[\theta_{i,\varepsilon,n}-\gamma_{n+1} H_i(\theta_{i,\varepsilon,n}, L_{T,n+1}^{\varepsilon}) \right]
\end{array}
\right.
\end{equation}
where $\Pi_K$ is the Euclidean projection onto the constraint set $K$, $H_1$ and $H_2$ are   given by relation (\ref{gradiant_V}) and the gain sequence $(\gamma_{n})_{n \geq 1}$ is a decreasing sequence of positive real numbers satisfying
\begin{equation}
 \sum_{n=1}^{\infty} \gamma_n = \infty \mbox{  and  } \sum_{n=1}^{\infty} \gamma_n^2 < \infty
\end{equation}
\begin{theorem}
\label{th:cv:CRM}
Let $i\in\{1,2\}$.  Assume $\PP(F_i(L_T)\neq 0) >0$,  $\PP(F_i(L_T^{\varepsilon})\neq 0) >0$ for all $\varepsilon>0$ and there exists $ a > 1 $  such that $ \mathbb{E}\left[F_i^{2a}(L_{T})\right]$ and $\sup_{\varepsilon>0}\mathbb{E}\left[F_i^{2a}(L^\varepsilon_{T})\right]$ are finite.
Let $K$ be a compact set such that $K\subset \mathring\Theta_{2a/(a-1)}$ then the following assertions hold.
\begin{itemize}
 \item If the unique $\theta^{*}_{i}=\displaystyle \argmin_{\theta\in \Theta_{i,3}}v_{i}(\theta)$ satisfies ${\theta}^{*}_{i}\in K$  then the sequence $\theta_{i,n} \underset{n\rightarrow +\infty}{\longrightarrow} \theta^{*}_{i}$ $a.s.$
 
 \item  If the unique $\theta^{*}_{i,\varepsilon}=\displaystyle\argmin_{\theta\in \Theta_{i,3}^{\varepsilon}}v_{i,\varepsilon}(\theta)$ satisfies ${\theta}^{*}_{i,\varepsilon}\in K$ then the sequence $\theta_{i,\varepsilon,n} \underset{n\rightarrow +\infty}{\longrightarrow} \theta^{*}_{i,\varepsilon}$ $a.s.$
\end{itemize}
\end{theorem}
\begin{proof} Both items can be proved in the same way, so we choose to give the proof only for the first one.
According to Theorem A.1. in Laruelle, Lehalle and Pagès \cite{LaruellepagesLehalle} on truncated Robbins Monro algorithm (see also Kushner and Yin \cite{KushnerYin} for more details):
in order to prove that $\theta_{i,n}^{\varepsilon} \underset{n\rightarrow +\infty}{\longrightarrow} \theta^{*}_{i,\varepsilon}$ $a.s.$, we need to check firstly the mean-reverting property, namely
\begin{equation*} 
 \forall \theta \neq \theta_{i}^{*} \in K, \quad \langle \nabla v_{i}(\theta), \theta- \theta_{i}^{*} \rangle >0.
\end{equation*}
This is satisfied using $\nabla v_{i}({\theta}^{*}_{i} )=0$  and  the convexity of $v_{i}$ ensured by Proposition \ref{convexity_v}.
Secondly, we have to check the non explosion assumption given by
$$
\exists C>0 \text{ such that } \forall \theta\in K, \quad \EE\left[|H_i(\theta, L_{T})|^2\right]< C(1+|\theta|^2).
$$
In fact, using Hölder's inequality with the couple $a$ and $a/(a-1)$, we obtain
$$
 \EE | H_i(\theta, L_T)|^{2} 
\leq  \EE^{\frac{1}{a}} \left[  F_i^{2a}(L_T)  \right] \EE^{\frac{a-1}{a}} \left[ |T\nabla \kappa(\theta) - L_T|^{2a/(a-1)} e^{-2 a/(a-1)  \theta. L_T}\right] e^{ 2T \kappa(\theta)} 
$$
Since $\mathbb{E}\left[F_i^{2a}(L_{T})\right]$ is finite and $\theta \in K \subset \Theta_{2a/(a-1)}$, we deduce that 
$\sup_{\theta \in K } \EE | H_i(\theta, L_T)|^{2} < \infty$ which completes the proof.
\end{proof}
\begin{theorem}
\label{th:cv:eps}
Considering the sequences given by relation \eqref{sequence_theta},
for $i\in \{1,2\}$, we have for all $n\in\NN$ 
$$\theta_{i,\varepsilon,n} \underset{\varepsilon \rightarrow 0}{\longrightarrow}\theta_{i,n}\quad a.s.$$
\end{theorem}
\begin{proof}
We proceed by induction. The base case is trivial and for the inductive step  we suppose that for $i\in\{1,2\}$,   $n\in\NN$, $\theta_{i,\varepsilon,n}$ converges to $\theta_{i,n}$ $a.s.$ as $\varepsilon$ goes to $0$
and we prove  the statement for $n+1$. We have  $\theta_{i,\varepsilon,n+1} = \Pi_{K} \left[\theta_{i,\varepsilon,n}-\gamma_{i+1} H_i(\theta_{i,\varepsilon,n}, L_{T,n+1}^{\varepsilon}) \right]$. 
By the continuity of the function $H_i$ given by (\ref{gradiant_V}), the almost sure convergence of $L_{T,n+1}^{\varepsilon}$ to $L_{T,n+1}$ and the continuity of the projection function $\Pi_{K}$, we deduce that $\theta_{i,\varepsilon,n+1}$ converges to $\theta_{i,n+1}$ $a.s.$ as $\varepsilon$ goes to $0$.
\end{proof}
The following corollary follows immediately  thanks to theorems  \ref{th:convergence},  \ref{th:cv:CRM} and  \ref{th:cv:eps}.
\begin{corollary}
\label{cor:cv:CRM}
Under assumptions of Theorem \ref{th:cv:CRM}, 
  the constrained  algorithm given  by routine \eqref{sequence_theta} satisfies for $i\in\{1,2\}$
 \begin{equation}
 \label{double_cv}
 \lim\limits_{\substack{\varepsilon \to 0 \\n \to \infty }}\theta_{i,\varepsilon,n} =\lim\limits_{\varepsilon\to 0}(\lim\limits_{n \to \infty}\theta_{i,\varepsilon,n})= \lim\limits_{n \to \infty}(\lim\limits_{\varepsilon\to 0}\theta_{i,\varepsilon,n})=\theta^*_i, \quad \mbox{$ \PP$-$a.s.$}
\end{equation}
\end{corollary}
\paragraph{Remark.}
Suppose for a while that we omit assumptions ${\theta}^{*}_{i}\in K$ and 
${\theta}^{*}_{i,\varepsilon}\in K $ in Theorem \ref{th:cv:CRM} above. 
According to Theorem 3.2. of Kawai \cite{Kawai} based on Theorem 2.1 of Kushner and Yin \cite{KushnerYin}
there exist $\bar \theta_{i}$ and $\bar \theta_{i,\varepsilon}$ in $K$ such that  $\theta_{i,n} \underset{n\rightarrow +\infty}{\longrightarrow} \bar \theta_{i}$ $a.s.$ and 
$\theta_{i,\varepsilon,n} \underset{n\rightarrow +\infty}{\longrightarrow} \bar \theta_{i,\varepsilon}$ $a.s.$
Moreover, $ v_{i}(\bar {\theta}_{i} )\leq v_{i} (\theta)$ and $ v_{i,\varepsilon}(\bar {\theta}_{i,\varepsilon} )
\leq v_{i,\varepsilon}(\theta )$ for all $\theta \in K$. In this case we can prove that the constrained  algorithm 
given  by routine \eqref{sequence_theta} satisfies   relation \eqref{double_cv} with $\bar {\theta}_{i,\varepsilon}$ instead of $\theta^*_i$.
\subsection{Central limit theorems}
In what follows, we consider the filtration  
$ {\mathcal{F}}_{T,k}=\sigma(L_{t,\ell}, L^{\varepsilon}_{t,\ell},0<\varepsilon<1,t \leq T,\ell \leq k)$,
where $(L_{\ell}, L^{\varepsilon}_{\ell})_{\ell\geq 1}$ are independent copies of $(L, L^{\varepsilon})$. Let us assume that 
there exists a family of  sequences $(\theta^\varepsilon_k)_{k\geq 0,0 < \varepsilon \leq 1 }$ and $(\theta_k)_{k\geq 0}$ satisfying
$$
(\mathcal H_{\theta})\quad
\left\{
\begin{array}{l}
\text{For each $\varepsilon>0$, $(\theta^\varepsilon_k)_{k\geq 0}$ and $(\theta_k)_{k\geq 0}$ are $({\mathcal{F}}_{T,k})_{k\geq 0}$-adapted}\\\\
\lim\limits_{k \to \infty}(\lim\limits_{\varepsilon \to 0}\theta_{k}^{\varepsilon})= \lim\limits_{k \to \infty} \theta_k=\lim\limits_{\varepsilon \to 0}(\lim\limits_{k \to \infty}\theta^{\varepsilon}_{k})= \lim\limits_{\varepsilon \to 0}\theta^*_\varepsilon=\theta^*, \quad \mbox{$\PP$-$a.s.$},
\end{array}
\right.
$$
with deterministic limits $\theta^*$ and $\theta_\varepsilon^*$.

At first, we start with studying the MC setting.  We use the adaptive importance sampling algorithm for the MC method to approximate our initial quantity of interest $\EE F(L_T)$ by 
\begin{equation}
\label{eq:ISMC}
 Q_{\varepsilon}^{\rm{ISMC}}=\frac{1}{N} \sum_{k=1}^{N}F(L_{T,k}^{\varepsilon, \theta_{k-1}^{\varepsilon}}) 
 e^{-\theta_{k-1}^{\varepsilon}. L_{T,k}^{\varepsilon,\theta_{k-1}^{\varepsilon}} + T \kappa_\varepsilon(\theta_{k-1}^{\varepsilon})}.
\end{equation}
Our task now is to establish a central limit theorem for the adaptive importance sampling Monte Carlo method (ISMC).
\begin{theorem}
\label{CLT adaptative Monte Carlo}
 Let $F:\RR^d\rightarrow \RR$ be a continuous function  satisfying assumption \eqref{weak_error} and such that $\sup_{0 < \varepsilon \leq 1}  \EE\left[ F^{2a} (L_{T}^{\varepsilon})\right] < +\infty$  for  $a>1$.
Moreover, assume that $\mbox{ Leb}(\Theta_q)>0 $ with $q>a/(a-1)$ and there exists a double indexed family 
$(\theta_{k}^{\varepsilon})_{k\in\NN,\varepsilon>0}$ satisfying $(\mathcal H_{\theta})$ and 
belonging to some compact subset  $K\subset \mathring \Theta_q$.  
Then, if we choose $N= \upsilon_\varepsilon^{-2} $, the following convergence holds
\begin{equation}
 \upsilon_\varepsilon^{-1} \left(Q_{\varepsilon}^{\rm{ISMC}} - \EE F(L_T) \right) \overset{\mathcal{L}}{\longrightarrow} \mathcal{N}(C_F,\sigma^2),\quad\mbox{ as } \varepsilon \searrow 0,
\end{equation}
where $\sigma^2:=\EE\left[F^2(L_{T}) 
 e^{-\theta^*.L_{T}^{\varepsilon} + T \kappa(\theta^*)}\right]-\left(\EE[F(L_{T})]\right)^2$.
\end{theorem}
\begin{proof}
By assumption \eqref{weak_error} we only need to study the asymptotic behavior of the martingale arrays $(M_{k}^{\varepsilon})_{k \geq 1}$ given by
$
 M_{k}^{\varepsilon}:= \upsilon_\varepsilon \sum_{i=1}^{k}\left( F(L_{T,i}^{\varepsilon, \theta_{i-1}^{\varepsilon}}) 
 e^{-\theta_{i-1}^{\varepsilon}. L_{T,i}^{\varepsilon,\theta_{i-1}^{\varepsilon}} + T \kappa_\varepsilon(\theta_{i-1}^{\varepsilon})} -\EE F(L_{T}^{\varepsilon})\right).
$
To do so, we plan to apply the Lindeberg-Feller central limit theorem for martingales arrays (see Theorem \ref{lindeberg martingales} in the Appendix section). The proof is divided into two steps.
\paragraph{Step 1.}
The quadratic variation of the martingale arrays $(M_{k}^{\varepsilon})_{k \geq 1}$ is given by 
\begin{equation}
\langle M^{\varepsilon}\rangle_{N}  = \frac{1}{N} \sum_{k=1}^{N} \EE \bigl[ F^2(L_{T,k}^{\varepsilon, \theta_{k-1}^{\varepsilon}}) 
 e^{-2\theta_{k-1}^{\varepsilon}. L_{T,k}^{\varepsilon,\theta_{k-1}^{\varepsilon}} + 2T \kappa_\varepsilon(\theta_{k-1}^{\varepsilon})}
  | \mathcal{ F}_{T,k-1} \bigr] - \left(\EE F(L_{T}^{\varepsilon}) \right)^{2}.
\end{equation}
Since $\theta_{k-1}^{\varepsilon}$ is $\mathcal{F}_{T,k-1}$-measurable and 
$(L_{T,k}^{\varepsilon, \theta})_{\theta\in \Theta_q} \independent\mathcal F_{T,k-1}$, by Esscher transform we obtain
\begin{equation*}
 \langle M^{\varepsilon}\rangle_{N}  = \frac{1}{N} \sum_{k=1}^{N} \gamma_{\varepsilon}(\theta_{k-1}^{\varepsilon})
 e^{T\kappa_\varepsilon(\theta_{k-1}^{\varepsilon})} - \left(\EE F(L_{T}^{\varepsilon}) \right)^{2},
\end{equation*}
where for all $\theta \in \Theta_q$, $ \gamma_{\varepsilon}(\theta) 
= \EE\left[  F^{2}(L_{T}^{\varepsilon}) e^{-\theta.L_{T}^{\varepsilon} } \right]$.
On the one hand, using assumption \eqref{weak_error}, we have $\lim_{\varepsilon\to 0}\EE F(L_{T}^{\varepsilon})=\EE F(L_{T})$. On the other hand,
thanks to relation \eqref{eq-UC-kapa} we have the uniform equicontinuity of the family $(\kappa_\varepsilon)_{\varepsilon>0}$ 
on the compact subset $K$.
So, we only need  to check this last property for the family $(\gamma_\varepsilon)_{\varepsilon>0}$ in view to use after that Lemma \ref{toeplitz} and then deduce  the convergence of $\langle M^{\varepsilon}\rangle_{N}$ toward $\gamma(\theta^*)-(\EE F(L_{T}))^2$ as $\varepsilon \searrow 0$, 
where $ \gamma(\theta) := \EE\left[  F^{2}(L_{T}) e^{-\theta.L_{T} } \right]$. 

Thus, it remains to prove the uniform equicontinuity of the family functions $(\gamma_{\varepsilon})_{\varepsilon >0}$ defined on the compact set $K$. Using H\"older's inequality and the assumption $\sup_{\varepsilon >0} \EE \left[ F^{2 a}(L_T^{\varepsilon})\right]<+\infty$, there exists  $c_1>0$ not depending on $\varepsilon$ such that 
\begin{eqnarray*}
 \left| \gamma_{\varepsilon}(\theta)-\gamma_{\varepsilon}(\theta')\right|   
 &\leq&  \EE \left[ F^{2}(L_T^{\varepsilon})  \bigl| e^{-\theta.L_T^{\varepsilon}} - e^{-\theta'.L_T^{\varepsilon}}  \bigr|\right] \\
   &\leq& c_1 \EE^{1/q} \left[   \bigl| e^{-\theta.L_T^{\varepsilon}} - e^{-\theta'.L_T^{\varepsilon}}  \bigr|^q\right].
\end{eqnarray*} 
By Taylor's expansion and standard calculations we easily get 
$$
|e^{- \theta.L_T^{\varepsilon}} - e^{- \theta'.L_T^{\varepsilon}}|^{q} \leq|\theta-\theta'|^{q}
\int_{0}^{1}
|L_T^{\varepsilon}|^{q} e^{-q(u\theta+ (1-u) \theta').L_T^{\varepsilon}}du.
$$
Therefore, we have
$$
\left| \gamma_{\varepsilon}(\theta)-\gamma_{\varepsilon}(\theta')\right| \leq c_1 |\theta-\theta'| \sup_{\theta \in \Theta_q} 
\EE^{1/q}\bigl[ |L_T^{\varepsilon}|^{q} e^{-q\theta.L_T^{\varepsilon}}\bigr].
$$
Hence, according to Lemma \ref{lem:ucv} there exists  a constant $c_2>0$ also not depending on and  $\varepsilon$ such that
\begin{equation}
\label{eq-UC-gamma}
\left| \gamma_{\varepsilon}(\theta)-\gamma_{\varepsilon}(\theta')\right| \leq c_2 |\theta-\theta'|.
\end{equation}
This completes the proof of the first step. 
\paragraph{Step 2.}
We check now the Lyapunov condition given by assumption {\it {B3}} in Theorem \ref{lindeberg martingales}.
So, let $\tilde a =\frac{aq+a}{2a+q}$, it is easy to check that $1<\tilde a <a$.  Once again using the mesurability properties of the family $(L_{T,k}^{\varepsilon, \theta})_{\theta\in \Theta_q}$ and the sequence $(\theta_{k}^{\varepsilon})_{k\geq 0}$, we get using the Esscher transform 
\begin{eqnarray*}
\sum_{k=1}^{N} 
\EE \left[ \left| M_{k}^{\varepsilon} - M_{k-1}^{\varepsilon} \right|^{2\tilde a} | \mathcal{F}_{T,k-1}\right] &=& \frac{1}{N^{\tilde a}} \sum_{k=1}^{N} 
\EE \Bigl[ \bigl|F(L_{T,k}^{\varepsilon, \theta_{k-1}^{\varepsilon}}) 
 e^{-\theta_{k-1}^{\varepsilon}. L_{T,k}^{\varepsilon,\theta_{k-1}^{\varepsilon}} + T \kappa_\varepsilon(\theta_{k-1}^{\varepsilon})}  -\EE F(L_{T}^{\varepsilon})  
\bigr|^{2\tilde a} |\mathcal{F}_{T,k-1}\Bigr]\\
  &\leq& \frac{2^{2\tilde a-1}}{N^{\tilde a}} \sum_{k=1}^{N} \gamma_{\tilde a,\varepsilon}(\theta_{k-1}^{\varepsilon})
  e^{  (2\tilde a-1) T \kappa_\varepsilon(\theta_{k-1}^{\varepsilon})} +\frac{2^{2\tilde a-1}}{N^{\tilde a}} \bigl|\EE F(L_{T}^{\varepsilon}) \bigr|^{2\tilde a} 
 \end{eqnarray*}
 where for all $\theta \in \Theta_q$, $ \gamma_{\tilde a,\varepsilon}(\theta) 
= \EE\left[  F^{2\tilde a}(L_{T}^{\varepsilon}) e^{-(2\tilde a-1)\theta.L_{T}^{\varepsilon} } \right]$. Then, by Hölder's inequality we get
$$
\gamma_{\tilde a,\varepsilon}(\theta)\leq  \EE^{\tilde a/a}\bigl[  F^{2 a}(L_{T}^{\varepsilon}) \bigr]\EE^{(a-\tilde a)/a}\bigl[   e^{-(2\tilde a-1)a/(a-\tilde a)\theta.L_{T}^{\varepsilon} } \bigr].
$$
Noticing that $q=(2\tilde a-1)a/(a-\tilde a)$, it results from assumption $\sup_{0 < \varepsilon \leq 1}  \EE\left[ F^{2a} (L_{T}^{\varepsilon})\right] < +\infty$ that 
 $ \gamma_{\tilde a,\varepsilon}$
is uniformly bounded on the compact subset $K\subset\Theta_q$. Moreover, using once again relation \eqref{eq-UC-kapa} we deduce the
uniform boundedness of the family $(\kappa_\varepsilon)_{\varepsilon>0}$ 
on the compact subset $K$. Hence, combining all these results together with assumption \eqref{weak_error}, we deduce the existence of $c_3>0$ 
not depending on $\varepsilon$ such that 
$
\sum_{k=1}^{N} \EE \left[ \left| M_{k}^{\varepsilon} - M_{k-1}^{\varepsilon} \right|^{2\tilde a} | \mathcal{F}_{T,k-1}\right]\leq \frac{c_3}{N^{\tilde a -1}}.
$
This completes the proof.
\end{proof}
\paragraph{Remark.}  If one have in mind to reduce the variance by using an  adaptive crude Monte Carlo method, it appears clear that the natural choice is
$$
\theta^{*}_{1}=\displaystyle \argmin_{\theta\in \Theta_{1,3}}v_{1}(\theta)\quad \text{and}\quad 
\theta^{*}_{1,\varepsilon}=\displaystyle\argmin_{\theta\in \Theta_{1,3}^{\varepsilon}}v_{1,\varepsilon}(\theta)\;\text{for}\; \varepsilon>0,
$$
where $v_1$ and  $v_{1,\varepsilon}$ are presented in section \ref{sec:IS}.
The construction of stochastic sequences converging almost surely to these desired targets and  satisfying $(\mathcal H_{\theta})$ is ensured by Corollary \ref{cor:cv:CRM}.

Now, we use the adaptive importance sampling statistical Romberg method (ISSR) to approximate our initial quantity of interest $\EE F(L_T)$ by 
\begin{multline}
\label{eq:ISSR}
 Q_{\varepsilon}^{\rm{ISSR}}:=\frac{1}{N_1} \sum_{k=1}^{N_1} F(L_{T,k}^{\varepsilon^{\beta}, \theta_{1,k-1}^{\varepsilon^{\beta}}}) e^{-\theta_{1,k-1}^{\varepsilon^{\beta}}.L_{T,k}^{\varepsilon^{\beta}, \theta_{1,k-1}^{\varepsilon^{\beta}}} + T \kappa_{\varepsilon^{\beta}}(\theta_{1,k-1}^{\varepsilon^{\beta}})}\\ +\frac{1}{N_2} \sum_{k=1}^{N_2} \left(F(L_{T,k}^{\varepsilon, \theta_{2,k-1}^{\varepsilon}}) - F(L_{T,k}^{\varepsilon^{\beta}, \theta_{2,k-1}^{\varepsilon}}) \right) e^{-\theta_{2,k-1}^{\varepsilon}.L_{T,k}^{\varepsilon, \theta_{2,k-1}^{\varepsilon}} + T \kappa_{\varepsilon}(\theta_{2,k-1}^{\varepsilon})}
\end{multline}
Our second result is a central limit theorem for the adaptive ISSR method
\begin{theorem}
\label{CLT:ISSR}
 Let $F:\RR^d\rightarrow \RR$ be a $\mathscr {C}^1$ function  satisfying assumption \eqref{weak_error}  and such that ${\sup_{0 < \varepsilon \leq 1} } \EE F^{2a} (L_{T}^{\varepsilon})$ and ${\sup_{0 < \varepsilon \leq 1} }\EE \left| \sigma^{-1}(\varepsilon) (F(L_{T}^{\varepsilon}) - F(L_{T}))\right|^{2a}$ are finite, for  $a>1$. Suppose also that the following assumptions are satisfied.  
\begin{itemize}
 \item [$\it H1.$] Condition (\ref{condition_convergence1}) in Theorem \ref{Cohen_Rosinski} holds and there exists a definite positive  matrix $\Sigma$  such that $\lim\limits_{\varepsilon \to 0}\sigma^{-2}(\varepsilon)\Sigma_{\varepsilon}=\Sigma$.
 \item [$\it H2.$] For $0<\beta <1$, we have $\lim\limits_{\varepsilon \to 0}\sigma(\varepsilon)\sigma^{-1}(\varepsilon^\beta)=0$ and  
 $\lim\limits_{\varepsilon \to 0}\veps\sigma^{-1}(\varepsilon^\beta)=0$.
\end{itemize}
Moreover, assume that $\mbox{ Leb}(\Theta_q)>0 $ with $q>a/(a-1)$ and for $i\in\{1,2\}$ there exists a double indexed family 
$(\theta_{i,k}^{\varepsilon})_{k\in\NN,\varepsilon>0}$  satisfying $(\mathcal H_{\theta})$ and 
belonging to some compact subset  $K_i\subset \mathring \Theta_q$.
If we choose $N_1=\upsilon_\varepsilon^{-2}$ and $N_2= \veps^{-2}\sigma^{2}(\varepsilon^\beta) $, then 
\begin{equation*}
 \veps^{-1} \left(Q_{\varepsilon}^{\rm ISSR} -\EE F(L_T) \right) \xrightarrow{\mathcal{L}} \mathcal{N}\left(C_{F}, 
 \sigma^2+ \tilde \sigma^2   \right), \quad \mbox{as  } \varepsilon  \rightarrow 0,
\end{equation*}
where 
$$\sigma^2= \EE \left[ F^{2}(L_T) e^{-\theta^{*}.L_T + T \kappa(\theta^{*})} \right] - \left[ \EE F(L_T) \right]^{2} \mbox{ and }\tilde \sigma^2= T\EE \left[ (\nabla F(L_T). \Sigma \nabla F(L_T))  e^{-\theta^{*}.L_T + T \kappa(\theta^{*}) }\right].$$
\end{theorem}

\begin{proof}
By assumption \eqref{weak_error} we only need to study the asymptotic  behavior of  
$
\veps^{-1}  Q_{1,\varepsilon}^{\rm ISSR} + \veps^{-1} Q_{2,\varepsilon}^{\rm ISSR}
$
with 
\begin{eqnarray*}
  Q_{1,\varepsilon}^{ \rm{ISSR}} & = &\frac{1}{N_1} \sum_{k=1}^{N_1} \Bigl( F(L_{T,k}^{\varepsilon^{\beta}, \theta_{1,k-1}^{\varepsilon^{\beta}}}) e^{-\theta_{1,k-1}^{\varepsilon^{\beta}}.L_{T,k}^{\varepsilon^{\beta}, \theta_{1,k-1}^{\varepsilon^{\beta}}} + T \kappa_{\varepsilon^{\beta}}(\theta_{1,k-1}^{\varepsilon^{\beta}})} - \EE F(L_{T}^{\varepsilon^{\beta}}) \Bigr)\\  &\mbox{\hspace{-1.8cm} and }&\\
  Q_{2,\varepsilon}^{\rm{ISSR}} & =& \frac{1}{N_2} \sum_{k=1}^{N_2} \Bigl( \bigl[F(L_{T,k}^{\varepsilon, \theta_{2,k-1}^{\varepsilon}}) - F(L_{T,k}^{\varepsilon^{\beta}, \theta_{2,k-1}^{\varepsilon}}) \bigr] e^{-\theta_{2,k-1}^{\varepsilon}.L_{T,k}^{\varepsilon, \theta_{2,k-1}^{\varepsilon}} + T \kappa_{\varepsilon}(\theta_{2,k-1}^{\varepsilon})} - \EE \bigl[ F(L_T^{\varepsilon}) - F(L_T^{\varepsilon^{\beta}}) \bigr]\Bigr).
\end{eqnarray*}
An application of Theorem \ref{CLT adaptative Monte Carlo} yields
$\veps^{-1} Q_{1,\varepsilon}^{\rm{ISSR}} \xrightarrow{\mathcal{L}} \mathcal{N} \left(0, \sigma^2 \right)$, as  $\varepsilon \rightarrow 0$.
For the second term, we aim to apply Theorem \ref{lindeberg martingales}. So, we introduce the martingale arrays $(M_k^{\varepsilon})_{k \geq 1}$
$$
 M_k^{\varepsilon}:= \frac{\veps^{-1}}{N_2} \sum_{\ell=1}^{k} \Bigl( \bigl(F(L_{T,\ell}^{\varepsilon, \theta_{2,\ell-1}^{\varepsilon}}) - F(L_{T,\ell}^{\varepsilon^{\beta}, \theta_{2,\ell-1}^{\varepsilon}}) \bigr) e^{-\theta_{2,\ell-1}^{\varepsilon}.L_{T,\ell}^{\varepsilon, \theta_{2,\ell-1}^{\varepsilon}} + T \kappa_{\varepsilon}(\theta_{2,\ell-1}^{\varepsilon})} - \EE \bigl[ F(L_T^{\varepsilon}) - F(L_T^{\varepsilon^{\beta}}) \Bigr]\Bigr).
$$
\paragraph{Step 1.}
Thanks to assumption $(\mathcal H_{\theta})$ and the Esscher transform, the quadratic variation of $M$ evaluated at $N_2$ is equal to
\begin{equation*}
  \langle M^{\varepsilon}\rangle_{N_2} = \frac{1}{N_2} \sum_{k=1}^{N_2}  
  \xi_{\varepsilon} (\theta_{2,k-1}^{\varepsilon})e^{T \kappa_{\varepsilon}(\theta_{2,k-1}^{\varepsilon})} -  \Bigl(\EE \bigl[\sigma^{-1}(\varepsilon^{\beta})( F(L_T^{\varepsilon})-F(L_T^{\varepsilon^{\beta}}) )\bigr]\Bigr)^2,
\end{equation*}
where for all $\theta \in \Theta_q$, 
$
\xi_{\varepsilon}(\theta) = \sigma^{-2}(\varepsilon^{\beta}) \EE \left( \left| F(L_T^{\varepsilon}) - F(L_{T}^{\varepsilon^{\beta}})\right|^{2} e^{-\theta.L_{T}^{\varepsilon}} \right).
$
Using the convergence in law given by relation \eqref{convergence}, the assumption 
${\sup_{0 < \varepsilon \leq 1} }\EE \left| \sigma^{-1}(\varepsilon) (F(L_{T}^{\varepsilon}) - F(L_{T}))\right|^{2a}<+\infty$
and the independence of $L_T$ and $W_T$, we deduce that the second term on the right hand side of the above equation vanishes when $\varepsilon$ tends to zero. Concerning  the first one, we aim to use Lemma \ref{toeplitz}. So, we only need to prove the  equicontinuity of the family $(\xi_\varepsilon)_{\varepsilon>0}$ on any compact subset of $\Theta_q$. First, we prove the simple convergence of $\xi_{\varepsilon}$
to $\xi$ with $
\xi(\theta) =  \EE \Bigl( \bigl|\nabla F(L_T). \Sigma^{\frac{1}{2}}W_T \bigr|^{2} e^{-\theta.L_{T}} \Bigr).
$
For this, we can proceed analogously to the proof of relation \eqref{convergence}. More precisely, we use  Taylor-Young's expansion with function $F$, the convergence in law given by \eqref{convergence}, the independence of $L_{T}^{\varepsilon}-L_{T}^{\varepsilon^{\beta}}$ and $L_{T}^{\varepsilon^{\beta}}$ and Slutsky's theorem to get 
\begin{equation*}
\sigma^{-2}(\varepsilon^{\beta}) \bigl| F(L_T^{\varepsilon}) - F(L_{T}^{\varepsilon^{\beta}})\bigr|^{2} e^{-\theta.L_{T}^{\varepsilon}}
\overset{\mathcal{L}} {\underset{\varepsilon \rightarrow 0}{\longrightarrow}}
   \bigl|\nabla F(L_T). \Sigma^{\frac{1}{2}}W_T \bigr|^{2} e^{-\theta.L_{T}}.
\end{equation*}
 Now,  applying Hölder's inequality with $\tilde a=\frac{aq}{a+q}$ yields
\begin{equation*}
 \EE \bigl|\sigma^{-2}(\varepsilon^{\beta}) \bigr( F(L_T^{\varepsilon}) - F(L_{T}^{\varepsilon^{\beta}})\bigl)^{2} e^{-\theta.L_{T}^{\varepsilon}} \bigr|^{\tilde a} \leq \EE^{\tilde a/a} \bigl|\sigma^{-1}(\varepsilon^{\beta}) \bigr( F(L_T^{\varepsilon}) - F(L_{T}^{\varepsilon^{\beta}})\bigl)  \bigr|^{ 2a}\EE^{(a-\tilde a)/a} e^{-\frac{\tilde a a}{a-\tilde a}\theta.L_{T}^{\varepsilon}}.
\end{equation*}
Using assumptions $\it {H2}$ and ${\sup_{0 < \varepsilon \leq 1} }\EE \left| \sigma^{-1}(\varepsilon) (F(L_{T}^{\varepsilon}) - F(L_{T}))\right|^{2a}<+\infty$, it is easy to check the uniform boundedness with respect to $\varepsilon$ of the first term on the right hand side of the above inequality.  Concerning the second one, since $q=\frac{\tilde a a}{a-\tilde a}$ we use relation \eqref{eq-UC-kapa} to deduce the same result. Hence, we have the simple convergence of $\xi_{\varepsilon}$ toward $\xi$ when $\varepsilon$ tends to zero. 
Therefore, it remains to prove the  equicontinuity of the family functions $(\xi_{\varepsilon})_{\varepsilon >0}$  on any compact subset $K\subset \Theta_q$. Replacing $F(L_T^{\varepsilon})$ by $\sigma^{-1}(\varepsilon^{\beta}) \bigr( F(L_T^{\varepsilon}) - F(L_{T}^{\varepsilon^{\beta}})\bigl)$ in the steps of the proof of relation \eqref{eq-UC-gamma} and using assumptions 
$\it {H2}$ and ${\sup_{0 < \varepsilon \leq 1} }\EE \left| \sigma^{-1}(\varepsilon) (F(L_{T}^{\varepsilon}) - F(L_{T}))\right|^{2a}<+\infty$ we prove 
the existence of a constant $c>0$ not depending on  $\varepsilon$ such that
\begin{equation}
\label{eq-UC-xi}
\left| \xi_{\varepsilon}(\theta)-\xi_{\varepsilon}(\theta')\right| \leq c |\theta-\theta'|.
\end{equation}
Thus, under assumption $(\mathcal H_{\theta})$, we get the almost sure convergence of $
 \xi_{\varepsilon}(\theta_{2,k}^{\varepsilon})$ toward $\xi(\theta^{*}) $ as $k$ goes to infinity and $\varepsilon$ vanishes.
We complete the proof of the first step using the almost sure convergence of $\kappa_{\varepsilon}(\theta_{2,k}^{\varepsilon})$ toward $\kappa(\theta^*)$ as $k$ goes to infinity and $\varepsilon$ vanishes. This last convergence is obtained thanks to relation \eqref{eq-UC-kapa}.
\paragraph{Step 2.} The second step of this proof consists on checking the Lyapunov condition {\it B3} of Theorem 
\ref{lindeberg martingales}. We proceed in the same way as in the second step of the proof of  Theorem \ref{CLT adaptative Monte Carlo}. We take $\tilde a =\frac{aq+a}{2a+q}$ and we get using the same arguments that
$\sum_{k=1}^{N_2} \EE \bigl[ \left| M_{k}^{\varepsilon} - M_{k-1}^{\varepsilon} \right|^{2\tilde a} | \mathcal{F}_{T,k-1} \bigr]$ is bounded by
$$
 \frac{2^{2\tilde a-1}}{N^{\tilde a}}\sum_{k=1}^{N_2}\xi_{a, \varepsilon}(\theta_{2,k-1}^{\varepsilon})e^{  (2\tilde a-1) T \kappa_\varepsilon(\theta_{2,k-1}^{\varepsilon})}
 +\frac{2^{2\tilde a-1}}{N^{\tilde a}}\Bigl| \EE\bigl[\sigma^{-1}(\varepsilon^{\beta}) (F(L_T^{\varepsilon}) - F(L_T^{\varepsilon^{\beta}})) \bigr]\Bigr|^{2a}
$$
where for all $\theta \in \Theta_q$, $ \xi_{\tilde a,\varepsilon}(\theta) 
= \EE\Bigl[ \bigl|\sigma^{-1}(\varepsilon^{\beta}) (F(L_T^{\varepsilon}) - F(L_T^{\varepsilon^{\beta}})) \bigr|^{2\tilde a}  e^{-(2\tilde a-1)\theta.L_{T}^{\varepsilon} }\Bigr] $. Then replacing  $F(L_T^{\varepsilon})$ by $\sigma^{-1}(\varepsilon^{\beta}) \bigr( F(L_T^{\varepsilon}) - F(L_{T}^{\varepsilon^{\beta}})\bigl)$ in the second step of the proof of Theorem \ref{CLT adaptative Monte Carlo}, the same arguments remain valid thanks 
to assumptions $\it {H2}$ and ${\sup_{0 < \varepsilon \leq 1} }\EE \left| \sigma^{-1}(\varepsilon) (F(L_{T}^{\varepsilon}) - F(L_{T}))\right|^{2a}<+\infty$.
So, we deduce the existence of $c>0$ 
not depending on $\varepsilon$ such that 
$$
\sum_{k=1}^{N_2} \EE \left[ \left| M_{k}^{\varepsilon} - M_{k-1}^{\varepsilon} \right|^{2\tilde a} | \mathcal{F}_{T,k-1}\right]\leq \frac{c}{N_2^{\tilde a -1}}.
$$
This completes the proof.
\end{proof}
\paragraph{Remark.}  Similarly as in the MC case, we still have in mind 
to reduce the variance associated now to  the SR method. This goes back to optimize separately $v_{1}$ and
$v_{2}$. Hence, the optimal choice corresponds to 
$$
\theta^{*}_{i}=\displaystyle \argmin_{\theta\in \Theta_{1,3}}v_{i}(\theta)\quad \text{and}\quad 
\theta^{*}_{i,\varepsilon}=\displaystyle\argmin_{\theta\in \Theta_{i,3}^{\varepsilon}}v_{i,\varepsilon}(\theta)\;\text{for}\; \varepsilon> 0\;\mbox{ and } i\in\{1,2\},
$$
where $v_i$ and  $v_{i,\varepsilon}$ are presented in section \ref{sec:IS}.
In the same way, the construction of stochastic sequences converging almost surely  to these desired targets and  satisfying $(\mathcal H_{\theta})$ is ensured by Corollary \ref{cor:cv:CRM}.
\section{Numerical results}
\label{num:test}
Now, we present numerical simulations that illustrate the efficiency of the ISSR method throughout the pricing of  vanilla options with an underlying asset following an exponential pure jump CGMY model. 
The CGMY process has been introduced by Carr, Geman, Madan and Yor  \cite{CGMY} with the aim to develop 
a model for the dynamic of equity log-returns which is rich enough to accommodate jumps
of finite or infinite activity, and finite or infinite variation. 
Monte Carlo  simulation of the CGMY process has been tackled in the literature 
specifically by Madan and Yor \cite{MadanYor}, Poirot and Tankov \cite{PoirotTankov} and Rosinski 
\cite{Rosinski}.  
A CGMY process 
is a pure jump process with  generating triplet $(0, 0, \nu)$ where for $C>0, G> 0, M > 0$ and $Y<2$  
\begin{equation}
\label{eq:CGMY}
 \nu(dx)= C \frac{e^{-Mx}}{x^{1+Y}} \mathbf{1}_{x>0} dx+\frac{C e^{-G |x|}}{|x|^{1+Y}}  \mathbf{1}_{x<0} dx.
\end{equation}

Following the notations of \cite{PoirotTankov}, we consider the Lévy-Kintchine 
representation with a truncation function $h$ and a characteristic exponent given by 
$$\psi(u)= i\gamma_h u+\int_{\RR}(e^{iux}-1-iuh(x))\nu(dx)\mbox{ with } 
\gamma_h=\int_{\RR}(h(x)-x\mathbf 1_{\{|x|\leq 1\}})\nu(dx),\; u\in\RR.$$
\begin{itemize}
 \item [$\bullet$] For $1<Y<2$ and  $h(x)=x$, we have  $\gamma_h=\int_{|x| \geq 1}x  \nu(dx)$ and
\begin{equation*}
\psi(u)= iu \gamma_h + C \Gamma(-Y) \left[ M^Y \left( (1-\frac{iu}{M})^{Y} -1 + \frac{iuY}{M} \right)+ G^Y \left((1+\frac{iu}{G})^{Y} -1 - \frac{iuY}{G} \right) \right]. 
\end{equation*}
\item[$\bullet$] For $0<Y<1$ and  $h(x)=0$, we have  $\gamma_h=\int_{|x| \leq 1} x  \nu(dx)$ and
\begin{equation*}
\psi(u)= iu \gamma_h + C \Gamma(-Y) \left[ M^Y \left( (1-\frac{iu}{M})^{Y} -1 \right)+ G^Y \left((1+\frac{iu}{G})^{Y} -1 \right) \right].
\end{equation*}
\end{itemize}
In what follows, we consider the risk neutral model with jumps generalizing the Black Scholes model by replacing the 
Brownian motion by $(L_t)_{0\leq t\leq T}$ the CGMY process with generating triplet $(\gamma, 0, \nu)$, $\gamma\in\RR$
and define the asset price 
$$S_t=S_0\exp(rt+L_t), \mbox{ where } r>0  \mbox{ is  the interest rate and } S_0>0.$$
To guarantee that $e^{-rt}S_t$ is a martingale we have to impose the condition
$\int_{|x|\geq 1}e^x\nu(dx)<\infty$ (which is satisfied as soon as $M>1$ ) and  the condition
\begin{equation}
\label{risk_neutral}
\gamma+\int_{\RR}(e^y-1-y\mathbf 1_{\{|y|\leq 1\}})\nu(dy)=0,
\end{equation}
or in other words  $\gamma=-\psi(-i)$.

Now, let us recall that for $0<\varepsilon <1$, the approximation $(L^{\varepsilon}_t)_{t\geq 0}$ 
of $(L_t)_{t\geq 0}$  is a Lévy process  with generating triplet $(\gamma,0,\nu_{\varepsilon})$ where
 $\nu_{\varepsilon}(dx):=\mathbf 1_{\{|x|\geq \varepsilon \}}\nu(dx)$.   It is worth to note
 that $(L^{\varepsilon}_t)_{t\geq 0}$ can be seen as a compound Poisson process with drift 
 $\gamma_{\varepsilon}:=\gamma -\int_{\varepsilon \leq |x|\leq 1}x\nu(dx)$, see \eqref{Levy_app}.
This compound Poisson process can be represented as the difference of two independent processes namely
the positive part and the negative one. More precisely, the positive part (resp. the negative part) 
is a compound Poisson process with jump size $\nu^+_{\varepsilon}=\mathbf 1_{\{x\geq \varepsilon \}}
\frac{\nu(dx)}{\nu([\varepsilon,+\infty[)}$ (resp. $\nu^-_{\varepsilon}=\mathbf 1_{\{x\leq -\varepsilon \}}
\frac{\nu(dx)}{\nu(]-\infty,-\varepsilon])}$ )  and intensity $\nu([\varepsilon,+\infty[)$ 
(resp. $\nu(]-\infty,-\varepsilon])$).   
To simulate these compound Poisson processes, we can use either the classical rejection method as described 
in Cont and Tankov \cite{ContTankov} or an improved method used by Madan and Yor \cite{MadanYor}.  
Indeed,  when we simulate the positive part  
we choose $\nu^+_{0,\varepsilon}$ so that 
$\frac{d\nu_{\varepsilon}^+}{d\nu^+_{0,\varepsilon}}(x)=e^{-Mx}\mathbf 1_{\{x>\varepsilon\}}\leq 1.$
Then, according to Rosinski \cite{Rosinski2001} we may simulate the paths of $\nu_{\varepsilon}^+$ from those of $\nu_{0,\varepsilon}^+$
by only accepting all jumps $x$ in the paths of $\nu_{0,\varepsilon}^+$ for which 
$\frac{d\nu_{\varepsilon}^+}{d\nu^+_{0,\varepsilon}}(x)>u$
where $u$ is an independent draw from uniform distribution. Hence, we use following algorithm 
\begin{algorithm}[H]
\caption{ Simulating the positive jump size $Z$ of the CGMY process  using Rosinski's rejection}
\begin{algorithmic} 
\REQUIRE $U_1$ and $U_2$ are uniform random variables and $Z= {\varepsilon}{U_1^{-\frac{1}{Y}}}$
\IF{$U_2 > \exp{-M.Z}$}
\STATE $Z=0$
\ENDIF
\RETURN Z

\end{algorithmic}
\end{algorithm}
In the same way, we simulate the negative jump part by replacing in the above algorithm the parameter $M$ by $G$.

Our aim is to test our approximation methods for computing the price of a vanilla option
with payoff $F$. To do so, we use  the importance sampling technique, introduced in section \ref{sec:IS},  to
approximate the price $e^{-r T} \EE F(S_T)$ by  
\begin{equation}
e^{-r T} \EE \left[F(S_T^{\varepsilon, \theta}) e^{-\theta. L_{t}^{\varepsilon, \theta}+ T \kappa_{\varepsilon}(\theta)} \right], 
\mbox{ with }S_T^{\varepsilon,\theta}=S_0\exp(rt+L^{\varepsilon,\theta}_t)
\end{equation}
where $L_{T}^{\varepsilon, \theta}$ is also a L\'evy process with generating triplet 
$(\gamma_{\varepsilon, \theta}, 0, \nu_{\varepsilon, \theta})$, 
where $\nu_{\varepsilon, \theta}=e^{\theta.x} \nu_{\varepsilon}(dx)$ and $\gamma_{\theta, \varepsilon}=\gamma_{\varepsilon}+ \int_{-1}^{1} x (e^{\theta.x}-1) \nu_{\varepsilon}(dx)$.
The choice of $\theta$ depends on using the classical 
MC method or the SR one. According to relation \eqref{optimal_theta}, 
$\theta_{1,\varepsilon}^{*}$ is the optimal choice for the MC method.
However, for the SR method, we omptize separately each quantity appearing in 
the associated variance and the optimal choice is given by the couple 
$(\theta_{1,\varepsilon}^{*},\theta_{2,\varepsilon}^{*})$ (see relation \eqref{optimal_theta}) . To compute these optimal terms,
we use the constrained algorithms introduced in the system \eqref{sequence_theta}.
It is worth to note that in practice it is easier to use $\kappa(\theta)$ instead of $\kappa_{\varepsilon}(\theta)$.
\subsection{One-dimensional CGMY process}
In this setting we consider the European call option with payoff $F(x)=(x-{\rm Strike})_+$ . 
The parameters of the CGMY model are chosen as follows: $S0=100, {\rm Strike}=100, C=0.0244, G=0.0765, M=7.5515, Y=1.2945$, the free interest rate $r=\log(1.1)$ and maturity time $T=1$.
We run $50000$ iteration for the constrained algorithm with the compact set $[-G,M]$. 
The obtained optimal values are given by $(\theta_{1,\varepsilon}^{*},\theta_{2,\varepsilon}^{*})=(5.3,2.5)$ 
(see Figure \ref{fig:optimal_theta}). 
\begin{figure} [H]
\begin{center}
\fbox{
      \includegraphics [width=0.8\textwidth,height=5cm] {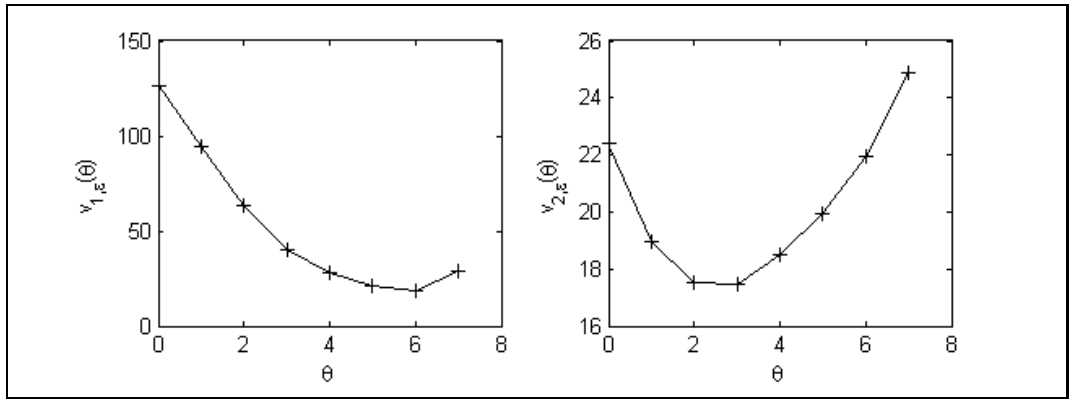}       
}
      \caption{\it Variances $v_{1,\varepsilon}$ and  $v_{2,\varepsilon}$ versus $\theta$ in the one-dimensional setting. }  

  \label{fig:optimal_theta}
   \end{center}

   \end{figure}  
   
In order to compare the ISMC algorithm \eqref{eq:ISMC} and the  ISSR one \eqref{eq:ISSR} we use the couple 
$(\theta_{1,\varepsilon}^{*},\theta_{2,\varepsilon}^{*})$ computed above. 
For this, we compute for each method the CPU time (per second) (the computations are done on a PC with a 2.5 GHz Intel core i5 processor) and an error measure given by the mean squared error (MSE) which is defined by
\begin{equation}
{\rm MSE}= \frac{1}{30} \sum_{i=1}^{30} (\mbox{Real value} - \mbox{Simulated value})^2. 
\label{MSE}
\end{equation}
The real value is obtained using the Fourier-cosine method introduced by Fang and Oosterlee \cite{FangOosterlee} for a one-dimensional  CGMY with an accuracy of order $10^{-10}$. This method is available  in the free online version of Premia platform (\url{https://www.rocq.inria.fr/mathfi/Premia/index.html}). For this setting, our ISSR algorithm  \eqref{eq:ISSR}  is now available in the latest premium version of Premia. 

For different values of $\varepsilon$, we give in Figure \ref{fig:MSE_time_CGMY_call} below the $\log$-$\log$ plot  of the obtained MSE versus the CPU time  for the classical Monte Carlo (MC), the statistical Romberg (SR), the importance sampling Monte Carlo (ISMC) and the importance sampling statistical Romberg (ISSR) methods. 

According to Table \ref{table:time reduction call} and for a fixed MSE of order  $6\cdot10^{-3}$, the ISSR method reduces the CPU time by a factor of $8,73$ compared to the ISMC one. 
\begin{figure} [hbt]
\begin{center}
      \includegraphics [width=14cm, height=9cm] {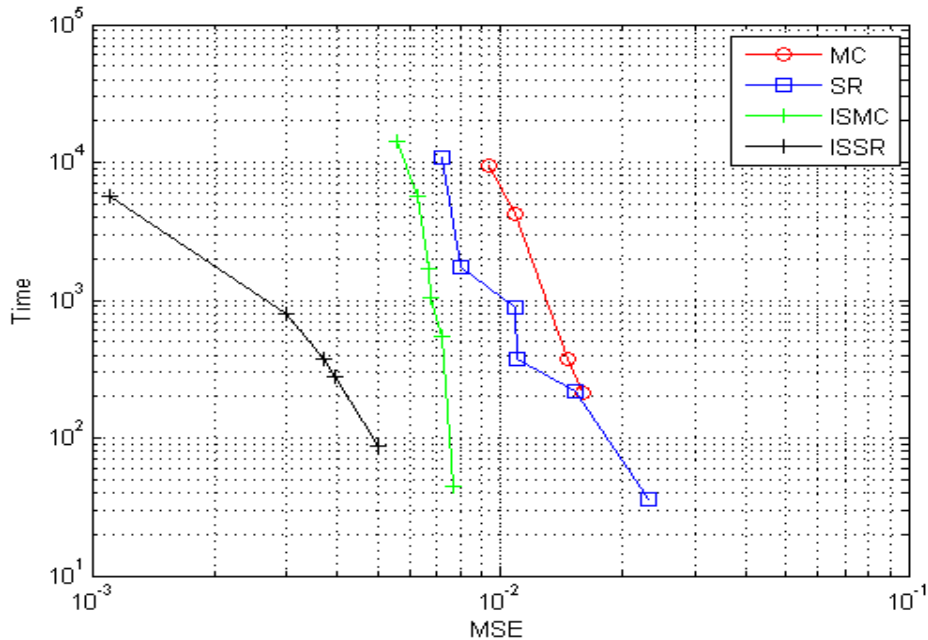}
  \caption{ CPU time versus MSE  in  the one-dimensional setting.}  
 \label{fig:MSE_time_CGMY_call}
 \end{center}
\end{figure} 
Clearly the ISSR method is the most efficient compared to the other ones. 
\begin{table}[H]
{\renewcommand{\arraystretch}{1} 
{\setlength{\tabcolsep}{0.6cm} 
\begin{center}
\begin{tabular} {*{3}{c}} 
    \hline 
    \multicolumn{3}{c}{\textbf{Time complexity reduction}} \\
    \hline
      MSE   & ISMC CPU time &  ISSR  CPU time  \\
    \hline
      $7\cdot10^{-3}$ & $7\cdot10^{2}$ & $5\cdot10^{2}$ \\

     $6,5\cdot 10^{-3}$ & $2\cdot10^{3}$ & $6\cdot10^{2}$ \\

     $6\cdot10^{-3}$ & $5,5\cdot10^{3}$ & $6,3\cdot10^{2}$ \\

     $5,5\cdot10^{-3}$ & $15\cdot10^{3}$ & $7\cdot10^{2}$ \\

     \hline
\end{tabular} 
 \caption{Time complexity reduction (ISSR versus ISMC).}
 \label{table:time reduction call}
\end{center} 
}}
\end{table}
\subsection{Two-dimensional CGMY process}
We focus now  on the computation of a price of the form  $e^{-rT}\EE F(S^1_T,S_T^2)$, where $F(x,y)=(x+y-{\rm Strike})_+$ and the couple $(S_t^1,S_t^2)_{0\leq t\leq T}$ denotes the underlying asset process. 
In this setting we choose $(S_t^1,S_t^2)=(S_0e^{ rt+L_t^1},S_0e^{ rt+L_t^2} )$ where 
$(L_t^1)_{0\leq t\leq T}$ and $(L_t^2)_{0\leq t\leq T}$ are two independent CGMY processes with generating triplets $(\gamma_1,0,\nu_1)$ and  $(\gamma_2,0,\nu_2)$ such that the processes $(e^{-rt}S_t^1)_{0\leq t\leq T}$ and $(e^{-rt}S_t^2)_{0\leq t\leq T}$ are two martingales. So,  it amounts to select$\gamma_1$ and $\gamma_2$ as in relation \eqref{risk_neutral}.

Since the Fourier-cosine method with high accuracy is no more available for the two-dimensional setting, the  "Benchmark" price is obtained by running the classical MC algorithm with a very small value of $\varepsilon$. Indeed, for $\varepsilon=10^{-6}$ the "Benchmark" price is $21.0782$ with a CPU time of $24718$ seconds.
The parameters of the considered two CGMY processes defined by $(C,G_1, M_1, Y)$ and $(C,G_2,M_2,Y)$ are chosen as follows: 
$C=0.0244, G_1=0.0765, M_1=7.55015, G_2=2, M_2=5, Y=0.9$,  $S0=100$,  ${\rm Strike}=200$, $r=\log(1.1)$ and the maturity time $T=1$.
Using the constrained algorithms \eqref{sequence_theta}, we obtain the values of the  optimal two-dimensional vectors given by relation \eqref{optimal_theta} and we get   
$\theta_{1,\varepsilon}^{*}=(4,3.5)$ and $\theta_{2,\varepsilon}^{*}=(3.5 , 1.1)$. In Figure \ref{fig:optimal_theta_2dim_Yinf1}, we plot the evolution of both variances $v_{1,\varepsilon}$ and $v_{2,\varepsilon}$ in terms of $\theta=(\theta_1,\theta_2)\in[-G_1,M_1]\times[-G_2,M_2]$. 
\begin{figure} [H]
\fbox{
      \includegraphics [width=0.97\textwidth,height=6cm] {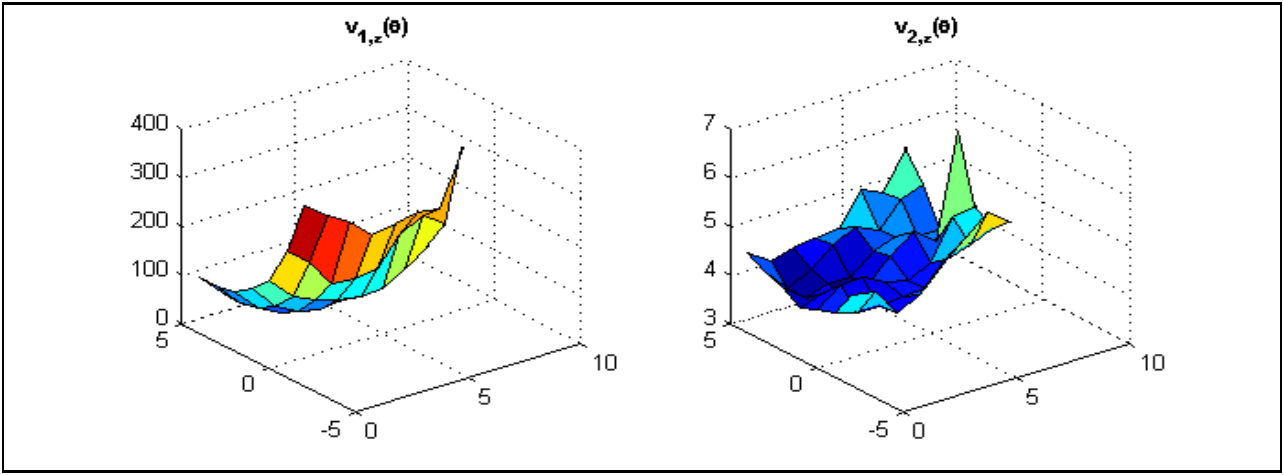}       
}
      \caption{ Variances $v_{1,\varepsilon}$ and  $v_{2,\varepsilon}$ versus $\theta$ in the two-dimensional setting. }  

  \label{fig:optimal_theta_2dim_Yinf1}

   \end{figure}  
Now we proceed as in the one-dimensional case to compare the different methods. Figure \ref{fig:MSE_time_CGMY_2dim_Yinf1} confirms the superiority of the ISSR
method over the other ones and this holds even when we compare it to the ISMC method.
Indeed, for a given  MSE, the ISSR spends less time than the other methods to compute the desired option price.  The difference in terms of computational time becomes more significant as soon as the MSE becomes very small, which corresponds to low values of $\varepsilon$ (see Figure \ref{fig:MSE_time_CGMY_2dim_Yinf1} below).

According to Table \ref{table:time reduction 2dim Yinf1} and  for a fixed  MSE of order $10^{-3}$, 
the ISSR  reduces the CPU time of the considered option price by a factor $2$ in comparison to the ISMC method. Moreover, this factor becomes more important when we consider a smaller MSE. In fact, for a fixed MSE of order $3\cdot 10^{-4}$, the ISSR reduces the CPU time by a factor $>5$ in comparison to the ISMC one.
\begin{table}[H]
{\renewcommand{\arraystretch}{1} 
{\setlength{\tabcolsep}{0.6cm} 
\begin{center}
\begin{tabular} {*{3}{c}} 
    \hline 
    \multicolumn{3}{c}{\textbf{Time complexity reduction}} \\
    \hline
      MSE  & ISMC  CPU time &  ISSR CPU time  \\
    \hline
      $10^{-3}$ & $40$ & $20$ \\

     $6\cdot 10^{-4}$ & $100$ & $30$ \\

     $4\cdot 10^{-4}$ & $250$ & $60$ \\

     $3\cdot 10^{-4}$ & $450$ & $80$ \\

     \hline
\end{tabular} 
 \caption{Time complexity reduction ISSR versus ISMC.}
 \label{table:time reduction 2dim Yinf1}
\end{center} 
}}
\end{table} 
\begin{figure} [H]
\begin{center}
      \includegraphics [width=14cm, height=9cm] {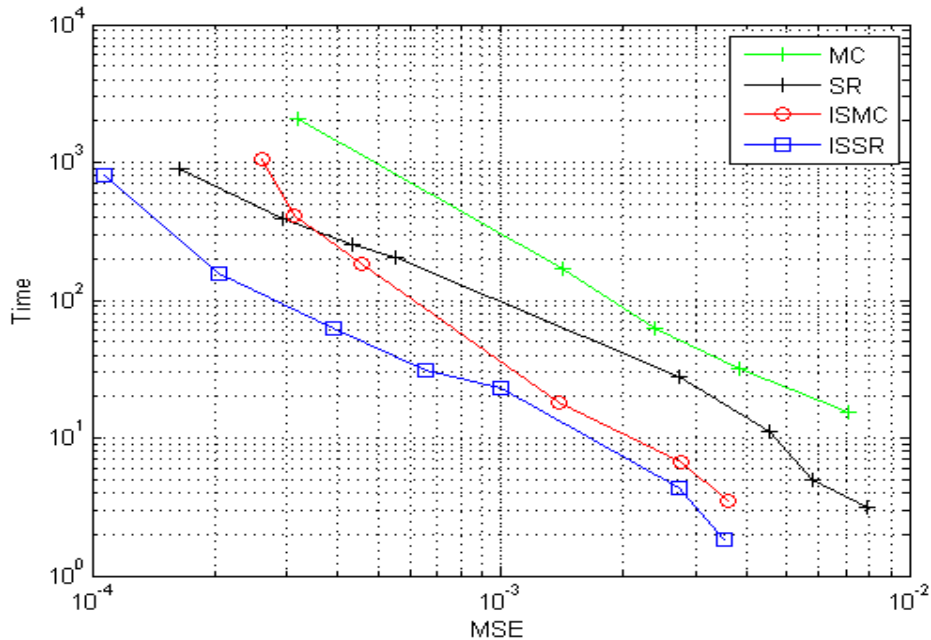}       
      \caption{ {\rm CPU} time versus {\rm MSE}  in  the two-dimensional setting.}  

  \label{fig:MSE_time_CGMY_2dim_Yinf1}
\end{center}
   \end{figure}  
\section{conclusion}
In this paper, we highlight the superiority of the ISSR 
method over the classical Monte Carlo approach for the setting of Lévy processes. It may be of interest to extend this study to the setting of Euler discretization schemes for Lévy driven diffusions developed by 
Protter and Talay \cite{ProtterTalay} and Jacod, Kurtz, Méléard and Protter \cite{JacodKurtzMeleardProtter}. 
Also, a next natural question consists on developing analogous results for path dependent options in 
exponential Lévy models in the spirit of the works of Dia and Lamberton \cite{DiaLambertonA, DiaLambertonB}. 
These two points will be the object of a forthcoming works. 
\section{Appendix}
We recall first the Lindeberg Feller Central Limit Theorem for independent random variables.
\begin{theorem}[Lindeberg Feller Central Limit Theorem \cite{Bil}]
\label{CLT Lindeberg Feller}
Let $(k_n)_{n\in\NN}$ be a sequence such that $k_n \longrightarrow \infty$, as $n \longrightarrow \infty$ and for each $n\in \NN$ we consider a sequence $X_{n1}, X_{n2}, ..., X_{n k_{n}}$ of independent centered and real square integrable random variables. We make the following two assumptions.
\begin{itemize}
 \item [${\it A1.}$] There exists a positive constant $v$ such that
 $\sum_{i=1}^{k_n} \EE (X_{ni})^2 \underset{n\rightarrow\infty}{\longrightarrow} v $.
 \item [${\it A2.}$] Lindeberg's condition holds: that is for all $\varepsilon >0$,
  $\sum_{i=1}^{k_n} \EE (|X_{ni}|^2 \mathbf{1}_{|X_{ni}|\geq \varepsilon}) \underset{n\rightarrow\infty}{\longrightarrow} 0$.
 Then 
 $$\sum_{i=1}^{k_n} X_{ni} \xrightarrow{\mathcal{L}} \mathcal{N}(0,v)\quad\mbox{ as }\; n\rightarrow\infty. $$ 
\end{itemize}
\end{theorem}
\paragraph{Remark.}
The following assumption  known as the Lyapunov condition implies the Lindberg's condition {\it A2.}.
\begin{itemize}
\item [${\it A3.}$] There exists a real number $a>1$ sucht that
$$ \sum_{k=1}^{k_{n}} \mathbb{E} \left[|X_{ni}|^{2a} \right] \underset{n\rightarrow\infty}{\longrightarrow} 0.$$
\end{itemize}
This result was generalized in the context of martingales arrays.
\begin {theorem} [Central Limit Theorem for martingales arrays \cite{DM}]
\label{lindeberg martingales}
Suppose that $(\Omega,\mathbb F, \PP)$ is a probability space and that for each $n$, we have a filtration $\mathbb{F}_{n}=(\mathcal{F}_{k}^{n})_{k\geq 0}$, a sequence $k_{n} \longrightarrow \infty  \mbox{ as  } n \longrightarrow \infty$  and a real square integrable vector martingale $M^{n}=(M_{k}^{n})_{k\geq 0}$ which is adapted to $\mathbb{F}_{n}$ and has quadratic variation denoted by $(\langle M \rangle_{k}^{n})_{k\geq 0} $.
We make the following two assumptions.
\begin{itemize}
 \item [B1.] There exists a deterministic symmetric positive semi-definite matrix $\varGamma$, such that
$$\langle M \rangle_{k_{n}}^{n}= \sum_{k=1}^{k_{n}} \mathbb{E}\left[ |M_{k}^{n}-M_{k-1}^{n}|^{2} | \mathcal{F}_{k-1}^{n}  \right] \overset{\mathbb{P}}{\underset{n\rightarrow\infty}{\longrightarrow}}  \varGamma.$$
 \item [B2.] Lindeberg's condition holds: that is, for all $\varepsilon>0$,
$$\sum_{k=1}^{k_{n}} \mathbb{E} \left[| M_{k}^{n}-M_{k-1}^{n}|^{2} 1_{\{| M_{k}^{n}-M_{k-1}^{n}|>\varepsilon\}} | \mathcal{F}_{k-1}^{n} \right] \overset{\mathbb{P}}{\underset{n\rightarrow\infty}{\longrightarrow}} 0.$$
Then $$ M_{k_{n}}^{n} \xrightarrow{\mathcal{L}} \mathcal{N}(0,\varGamma)\quad\mbox{ as }\; n\rightarrow\infty.$$ 
\end{itemize}
\end {theorem}
\paragraph{Remark.}
The following assumption  known as the Lyapounov condition, implies the Lindberg's condition {\it B2.},
\begin{itemize}
\item [{\it B3.}] There exists a real number $a>1$, sucht that
$$ \sum_{k=1}^{k_{n}} \mathbb{E} \left[| M_{k}^{n}-M_{k-1}^{n}|^{2a} | \mathcal{F}_{k-1}^{n} \right] \overset{\mathbb{P}}{\underset{n\rightarrow\infty}{\longrightarrow}}  0.$$
\end{itemize}
Moreover, we give a double indexed version of the Toeplitz lemma. For a proof of this result see 
Lemma 4.1 in \cite{BenalayaHajjiKebaier}
\begin{lemma}
\label{toeplitz}
Let $(a_{i})_{  1\leq i \leq k_{n} } $ a sequence of real positive numbers, where $k_{\varepsilon}\uparrow \infty$   as $\varepsilon$ tends to $0$, and $(x_{i}^{\varepsilon})_{i\geq 1,0 < \varepsilon \leq 1}$ a  double indexed sequence such that
\begin{itemize}
 \item[(i)] $\lim\limits_{\varepsilon \to 0} \sum_{ 1\leq i \leq k_{\varepsilon}} a_{i}=\infty$
 \item[(ii)]  $\lim\limits_{\substack{i \to +\infty \\ \varepsilon \to 0}} x_{i}^{\varepsilon}= \lim\limits_{i \to +\infty}(\lim\limits_{\varepsilon \to 0} x_{i}^{\varepsilon}) = \lim\limits_{\varepsilon \to 0}(\lim\limits_{i \to +\infty} x_{i}^{\varepsilon}) = x$
\end{itemize}
Then
$$\lim\limits_{\varepsilon \rightarrow 0} \frac{\sum_{i=1}^{k_{\varepsilon}} a_{i} x_{i}^{\varepsilon}}{\sum_{i=1}^{k_{\varepsilon}} a_{i}} =x.$$
\end{lemma}
   \input{IS_SR_Levy_processes.bbl}

\end{document}

%% file: IS_SR_Levy_processes.bbl
\def\cprime{$'$}